\documentclass[review]{elsarticle}

\usepackage{lineno,hyperref}
\modulolinenumbers[5]

\usepackage{amsmath}
\usepackage{amssymb}
\usepackage{amsthm}
\usepackage{mathtools}
\usepackage{graphicx}
\usepackage{caption}
\usepackage[]{algorithm2e}
\usepackage{bm}
\usepackage{enumitem}
\usepackage{xcolor}

\journal{Advances in Applied Mathematics}

\newcommand{\ie}{i.e.\ }
\newcommand{\eg}{e.g.\ }

\newcommand{\cf}{cf.\ }

\newcommand{\reals}{\mathbb{R}}
\newcommand{\integers}{\mathbb{Z}}

\newcommand{\isom}{\mathsf{ISO}}
\newcommand{\tran}{\mathsf{T}}
\newcommand{\rota}{\mathsf{R}}
\newcommand{\perm}{\mathsf{P}}
\newcommand{\nega}{\mathsf{N}}
\newcommand{\orth}{\mathsf{O}}

\newcommand{\hcal}{\mathcal{H}}
\newcommand{\tcal}{\mathcal{T}}
\newcommand{\rcal}{\mathcal{R}}
\newcommand{\ncal}{\mathcal{N}}
\newcommand{\pcal}{\mathcal{P}}
\newcommand{\acal}{\mathcal{A}}
\newcommand{\bcal}{\mathcal{B}}

\newcommand{\dist}{{\rm d}}
\newcommand{\idty}{{\rm id}}

\newcommand{\ambsp}{\integers^n}
\newcommand{\iz}{\isom(\ambsp)}
\newcommand{\tz}{\tran(\ambsp)}
\newcommand{\rz}{\rota(\ambsp)}
\newcommand{\pz}{\perm(\ambsp)}
\newcommand{\nz}{\nega(\ambsp)}
\newcommand{\onz}{\orth_n(\integers)}

\newcommand{\rhowl}{\rho{_\mathsf{WL}}}

\newcommand{\ginner}{\left[A_k\cdots A_1\right]}

\newcommand{\atomicsuper}{A_k \cup \cdots \cup A_1}
\newcommand{\sginner}[1]{\left[(\acal_k)_{#1} \cdots (\acal_1)_{#1}\right]}

\theoremstyle{definition}
\newtheorem{theorem}{Theorem}[section]
\newtheorem{definition}{Definition}[section]
\newtheorem{lemma}{Lemma}[section]
\newtheorem{corollary}{Corollary}[section]
\newtheorem{remark}{Remark}[section]

\begin{document}

\begin{frontmatter}

\title{Orbit Computation for Atomically Generated Subgroups of Isometries of $\ambsp$\tnoteref{mytitlenote}}
\tnotetext[mytitlenote]{This work was funded in part by the IBM-Illinois Center for Cognitive Computing Systems Research (C3SR), a research collaboration as part of the IBM AI Horizons Network; and in part by grant number 2018-182794 from the Chan Zuckerberg Initiative DAF, an advised fund of Silicon Valley Community Foundation.}

%% Group authors per affiliation:
%\author{Elsevier\fnref{myfootnote}}
%\address{Radarweg 29, Amsterdam}
%\fntext[myfootnote]{Since 1880.}

%% or include affiliations in footnotes:
\author[mymainaddress]{Haizi Yu\corref{mycorrespondingauthor}}
\ead{haiziyu7@illinois.edu}
\author[mysecondaryaddress]{Igor Mineyev}
\ead{mineyev@illinois.edu}
\author[mymainaddress]{Lav R. Varshney}
\ead{varshney@illinois.edu}

\cortext[mycorrespondingauthor]{Corresponding author}
\address[mymainaddress]{Coordinated Science Laboratory, University of Illinois at Urbana-Champaign,\\1308 W Main Street, Urbana, IL 61801, USA}
\address[mysecondaryaddress]{Department of Mathematics, University of Illinois at Urbana-Champaign,\\1409 W Green Street, Urbana, IL 61801, USA}

\begin{abstract}
Isometries are ubiquitous in nature; isometries of discrete (quantized) objects---abstracted as the group of isometries of $\ambsp$ denoted by $\iz$---are important concepts in the computational world.
In this paper, we compute various isometric invariances which mathematically are orbit-computation problems under various isometry-subgroup actions $H \curvearrowright \ambsp, H \leq \iz$.
One computational challenge here is about the \emph{infinite}: in general, we can have an infinite subgroup acting on $\ambsp$, resulting in possibly an infinite number of orbits of possibly infinite size.
In practice, we restrict the set of orbits (a partition of $\ambsp$) to a finite subset $Z \subseteq \ambsp$ (a partition of $Z$), where $Z$ is specified a priori by an application domain or a data set.
Our main contribution is an efficient algorithm to solve this \emph{restricted} orbit-computation problem in the special case of \emph{atomically generated subgroups}---a new notion partially motivated from interpretable AI.
The atomic property is key to preserving the \emph{semidirect-product structure}---the core structure we leverage to make our algorithm outperform generic approaches.
Besides algorithmic merit, our approach enables \emph{parallel-computing} implementations in many subroutines, which can further benefit from hardware boosts.
Moreover, our algorithm works efficiently for \emph{any} finite subset ($Z$) regardless of the shape (continuous/discrete, (non)convex) or location; so it is application-independent.
\end{abstract}


\end{frontmatter}

%\linenumbers

\section{Introduction}
\label{sec:introduction}

An \emph{isometry} of a metric space is an important concept in geometry~\cite{Coxeter1969}.
Isometries of various kinds are ubiquitous in the world, and are embedded as an innate preference in biological perception~\cite{Goldman1986,Ullman1979}.
In light of digital computers, isometries of discrete/quantized objects are widely studied in the computational modeling of real-world data observed in different human perception modalities.
Examples include vision (computer graphics and animations~\cite{Carlson2017}), audition (computer music~\cite{Tymoczko2010}), motion and kinematics (robotics~\cite{StramigioliB2001}), and experimental science (crystallography~\cite{HahnSA1983,Bieberbach1911}, physics~\cite{Noether1915}, and biology~\cite{YuVS2019,ClarkSSCDSSHRJRO2019}).
These studies are concerned with not only isometry classifications, but also the symmetries induced by various subclasses of isometries (\ie isometry subgroups).
Isometry-induced symmetries, among many other types of symmetries, are strongly connected to invariance theory~\cite{DerksenK2015,Olver1995}, and are key to computational abstraction wherein the abstracted concepts are invariant under the considered isometry subgroup~\cite{YuMV2019}.
Mathematically, the symmetry induced by any isometry subgroup is represented by \emph{orbits} under the corresponding isometry subgroup action.
So, it is important to have an algorithm that efficiently and explicitly computes the orbits.

However, computing orbits can be challenging when both the subgroup and the space are infinite.
In this case, the desired output---the set of orbits---is a partition of the infinite space, which immediately poses the question of how to represent this partition realistically and if possible, explicitly.
To address this question, we present only \emph{part of} the desired output: the partition restricted to some finite subset of the whole space where the finite subset is usually determined beforehand by an actual application domain or a data set.
Even though the representation of the output now becomes finite (and thus realistic), in order to compute it, we almost always have to use the whole subgroup acting on the whole infinite space.
The stopping criterion for exhausting an orbit is unclear when the full orbit is infinite.
In the worst case, orbit computation for an infinite group acting on an infinite space is unsolvable: the famous \emph{word problem} can be cast as an orbit computation problem~\cite{YuMV2019, Novikov1955, Boone1958, Britton1958}.

In this paper, for a subgroup $H$ (possibly infinite) of the group of isometries of $\ambsp$, we compute orbits for the action $H \curvearrowright \ambsp$.
Further, to derive a highly efficient and specialized algorithm for computing restricted orbits, we focus on a special type of subgroups, namely \emph{atomically generated subgroups}---a new notion we introduce in the paper.
Such subgroups have a canonical \emph{semidirect-product decomposition} which is the essential structure we leverage to make our specialized algorithm outperform generic ones.
Besides pure algorithmic merit, our approach enables \emph{parallel-computing} implementations in many subroutines, which can further benefit from hardware boosts (\eg a GPU/TPU implementation).
In addition, compared to implicit methods and/or approximate methods, our algorithm outputs a precise and explicit form of the orbits restricted to any finite $Z \subseteq \ambsp$, and the efficiency of the algorithm is not sensitive to $Z$: no matter whether $Z$ is continuous or discrete, convex or not, and no matter where $Z$ lies in $\ambsp$.
The fact that our algorithm can efficiently handle any finite subset $Z$ demonstrates a \emph{universal applicability} of the algorithm since $Z$ is predetermined by any particular application (which is outside our control).

\section{Problem Preview, Related Work, and Our Contributions}
\label{sec:problem-preview-related-work-and-our-contributions}

Let us preview the problem we want to solve in general (see a more detailed, step-by-step derivation in Section~\ref{sec:the-problem-of-orbit-computation}):
\begin{equation}
\label{eqn:problem-sneak-peek}
\begin{aligned}
\mbox{\textbf{Inputs:}} \quad & \mbox{1) any finite } S \subseteq \iz; \\
& \mbox{2) any finite } Z \subseteq \ambsp. \\
\mbox{\textbf{Output:}} \quad & Z/\langle S \rangle := (\ambsp/\langle S \rangle)|_Z := \{\hspace{0.03in}(\langle S \rangle \cdot x) \cap Z \mid x \in Z\hspace{0.03in}\}.
\end{aligned}
\end{equation}
Note that Problem~\eqref{eqn:problem-sneak-peek} is computationally well-defined, meaning the inputs as well as the desired output are finite objects.
Any real application of the problem predetermines the finite subset $Z \subseteq \ambsp$, and for any partition $\pcal$ of $\ambsp$, the notation $\pcal|_Z := \{P\cap Z \mid P \in \pcal\} \backslash \emptyset$ denotes the partition restricted to $Z$.
As required by most applications, we want the output $Z/\langle S \rangle$ to be in a form that \emph{explicitly} gives the partition of $Z$: for every cell in the partition, list all elements in the cell;
or alternatively, for every element in $Z$, specify which cell it belongs to.
In contrast, an impractical formula of some computable entities (\eg an infinite union of finite sets or the linear span of a finite basis) is \emph{not} satisfactory.

Problem~\eqref{eqn:problem-sneak-peek} has close connections to many problems in computational group theory.
One might try to either directly apply some existing and more generic techniques to solve Problem~\eqref{eqn:problem-sneak-peek} as a special case, or instead decompose Problem~\eqref{eqn:problem-sneak-peek} into subproblems and borrow different techniques to solve the subproblems individually.
However, by the end of this section, we will conclude: to the best of our knowledge, no existing algorithm can solve Problem~\eqref{eqn:problem-sneak-peek} directly;
hence our central contributions are in defining a special case of Problem~\eqref{eqn:problem-sneak-peek} where the problem can be decomposed such that for each subproblem we are able to solve it by either borrowing a state-of-the-art technique or designing a new approach that outperforms the state-of-the-art.

\subsection{No Existing Algorithm Works for the Whole Problem Directly}
\label{sec:none-existing-algorithm-work-directly}
Problem~\eqref{eqn:problem-sneak-peek} by nature is an orbit computation problem under some isometry subgroup action.
It is a special case of an orbit computation problem in general;
further, the subgroups are special cases of affine crystallographic groups.

\begin{enumerate}[label=$\blacksquare~$]
\item \textbf{Generic Orbit-Computation Algorithms Cannot Be Used.}
As a special case of a generic orbit computation problem, one may try directly using a generic algorithm, \eg $\textsc{Orbit}(x,S)$ from Chapter\ 4.1 in \cite{HoltEO2005}, to get the orbit of $x$.
However, $\textsc{Orbit}(x,S)$ and similar algorithms are not directly applicable to Problem~\eqref{eqn:problem-sneak-peek}, since they are designed for a finite orbit but \emph{not} for a possibly infinite orbit restricted to any finite subset.

Attempts to apply/modify algorithms like $\textsc{Orbit}(x,S)$ to solve our restricted orbit computation Problem~\eqref{eqn:problem-sneak-peek} will encounter one of two difficulties:
a) running such an algorithm in the full scope $\ambsp$ is endless when the full orbit is infinite;
b) running such an algorithm only in the local scope $Z$ will halt but the result is incorrect in general.
The risk in the latter case comes from the fact that the underlying group action is still $H \curvearrowright \ambsp$ even though every instance of Problem~\eqref{eqn:problem-sneak-peek} only computes a partition of $Z$ (there is \emph{no} guarantee that $H$ will act on $Z$).
One may fail to discover $x' \sim x$ if every path $x \xmapsto{s_1} \cdots \xmapsto{s_k} x'$ (for $s_1, \ldots, s_k \in S$) has to go outside $Z$.
For example, consider $Z = \{0,1\}^2 \subseteq \integers^2$, and $S = \{t_{\bm 1}, r_{-I}\}$ where $t_{\bm 1}: \integers^2 \to \integers^2$ is a translation defined by $t_{\bm 1}(x) := x + {\bm 1} = x +(1,1)$ and $r_{-I}: \integers^2 \to \integers^2$ is a negation defined by $r_{-I}(x) := -Ix = -x$.
One can check: from $(1,0)$ you can go nowhere, since applying either generator in $S$ will go outside the scope $Z$.
However, $(1,0)$ and $(0,1)$ should be in the same orbit, since $(1,0) \xmapsto{r_{-I}} (-1,0) \xmapsto{t_{\bm{1}}} (0,1)$ but this path is not entirely in $Z$ (particularly, $(-1,0) \notin Z$).

\underline{Note:} a safer way is to discover orbit relation in some enlarged finite superset $Z^+ \supseteq Z$ and restrict the result back to $Z$ in the last step (clearly, the safest way is to consider $Z^+ = \ambsp$ which however is impractical).
Nevertheless, how much one should enlarge $Z$ depends on $S$ and $Z$ (particularly the shape and the location of $Z$ in $\ambsp$); unfortunately in the worst case, $Z^+$ has to approach the entire $\ambsp$.

\item \textbf{Generic Space-Group Algorithms Cannot Be Used.}
The group $\iz$ is a special space group, and every subgroup $H \leq \iz$ is a special affine crystallographic group (whose translation subgroup is not necessarily full rank).
So, one may try using Cryst---a GAP4~\cite{gap4} package which solves computational problems about affine crystallographic groups.
However, as noted in its manual~\cite{EickGN2019}, ``For infinite (affine crystallographic) groups, some restrictions apply.
For instance, algorithms from the orbit-stabilizer family can work only if the orbits generated are finite.''

One topic closely related to orbit computation for possibly infinite affine crystallographic groups is computing the so-called Wyckoff positions~\cite{EickGN1997} (also included in Cryst).
However, computing Wyckoff positions is insufficient for computing orbits, since the Wyckoff positions only give a coarser partition compared to the set of orbits: points from different orbits may be in the same Wyckoff position.
For example, consider $H = \{t_v \mid v \in (2\integers)^2\} \cong (2\integers)^2$, a space group of translations only (so its point group is trivial).
$\mbox{Stab}_H(x)$ is trivial for any $x \in \reals^n$, so there is only one Wyckoff position of $H$ (namely $\reals^n$).
However, $(0,0)$ and $(1,1)$ are not in the same orbit.
Further, even when a Wyckoff position coincides with an orbit, the existing algorithms only return an affine subspace $A$ and use $A$ to indirectly describe the Wyckoff position through the (possibly infinite) union $\bigcup_{g\in G}g(A)$.
It is unclear if there is an efficient way of further converting this output to our desired output, which requires explicitly listing every point in the intersection of the orbit and the finite subset $Z$.

One may potentially make another connection to polycyclic affine crystallographic groups.
However, not all affine crystallographic groups are polycyclic, especially in high-dimensional spaces (\cf Lemma~8.30 in \cite{HoltEO2005}).
Our Problem~\eqref{eqn:problem-sneak-peek} is \emph{not} restricted to low-dimensional spaces; for example, the point group that is isomorphic to the alternating group $A_5$ (simple and non-cyclic) is not polycyclic.
\end{enumerate}

\subsection{Contributions: What We Do and What We Do Not}
\label{sec:contributions}

Our top-level contribution is to introduce a new, special condition---namely the \emph{atomic} condition---to Problem~\eqref{eqn:problem-sneak-peek}.
Under this condition, the whole problem can be decomposed recursively based on semidirect-product structures into subproblems that are easy to solve.
To solve each subproblem, we either plug in a state-of-the-art algorithm (in which case we do not claim a contribution) or propose a new algorithm that outperforms the state-of-the-art (if any and if applicable, in which case we claim a contribution).
We list our contributions more explicitly as follows.
\begin{enumerate}[label=$\blacksquare~$]
\item \textbf{The Atomic Condition and Semidirect-Product Decomposition (Contribution)}.
The core idea in decomposing Problem~\eqref{eqn:problem-sneak-peek} is based on a nested semidirect-product structure of the whole group of isometries: $\iz$ decomposes into translations and rotations; rotations further into negations and permutations (their definitions are detailed later).
However, not all subgroups $H \leq \iz$ share such a structure.
So, we introduce \emph{atomically generated subgroups} which we prove to inherit the global structure.
Our work is then to efficiently solve Problem~\eqref{eqn:problem-sneak-peek} in the special case of atomically generated subgroups (\ie the finite $S$ is further atomic).

For the same reasons, one can check: the existing algorithms mentioned in Section~\ref{sec:none-existing-algorithm-work-directly} are still not directly applicable when restricting Problem~\eqref{eqn:problem-sneak-peek} to the above special case.
Comparing to other structures widely-leveraged in computational group theory, we will show that the semidirect-product structure in this paper is dual to a polycyclic structure and is stronger than a decomposition of normal subgroup and its quotient; accordingly, being an atomically generated subgroup of $\iz$ is different from being polycyclic and is stronger than being affine crystallographic in general.

\item \textbf{Subproblem Regarding Translations (Partial Contribution).}
In this part, we compute orbits under the translation subgroup (of $\langle S \rangle$) only.
As a sketch, we first seek a basis of the translation subgroup and then use the basis to compute the orbits.
Our contribution in solving this subproblem is to derive an explicit formula that computes an orbit label for every $x \in Z$.
Orbit computation via this formula is highly efficient since it enjoys three computational benefits.
First, the formula is explicit and analytic as opposed to an iterative process.
Second, the formula involves only matrix operations which are nowadays much more efficient than other types of computational operations; this is not purely from the algorithmic perspective but also from recent advances in hardware (\eg GPUs and TPUs).
Third, applying the formula to every $x \in Z$ is independent (\ie it does not depend on the result from any other $x' \in Z$); this particularly means that a) the computation for every $x \in Z$ can be parallelized and b) we can handle any finite $Z \in \ambsp$ regardless of its shape or location.

We do \emph{not} claim any contribution in computing a basis from a generating set, which is solved by directly using the LLL (Lenstra-Lenstra-Lov{\'a}sz) algorithm originally designed for Hermitian-normal-form calculation~\cite{LenstraLL1982}.

\item \textbf{Subproblem Regarding Rotations (Partial Contributions).}
In this part, we further merge the orbits computed under the translation subgroup by using the rotation subgroup (of $\langle S \rangle$).
As a sketch, we merge two orbits if there exist two points---one from each orbit---related by a rotation in the rotation subgroup.
Our contribution in solving this subproblem is to efficiently compute any atomically generated rotation subgroup.
We once again leverage the semidirect-product structure but at a smaller scale regarding the negation-permutation splitting of rotations.
Note that computing rotation subgroups is not hard in general since they are matrix groups and the whole group of rotations of $\ambsp$ is finite.
However, our approach outperforms more generic algorithms, \eg computational methods for matrix groups over finite fields~\cite{BaarnhielmHLGO2015}.
This is because in our special case, we particularly have the atomic property and the resulting semidirect-product structure to leverage;
whereas in general, the structure is not immediate for free but must be discovered using additional computations (\eg construct a composition tree/series which is not needed in our case).

We do \emph{not} claim any contribution in computing permutation subgroups; we directly plug in any off-the-shelf solver instead.
In special cases when the dimensionality $n$ is small, we precompute---a one-time computation---and cache the full family of all subgroups of the symmetric group ${\rm Sym}_n$ in the memory and/or hard drive, and use the cached data later.

\end{enumerate}

\section{Semidirect Product: Review and Generalization}
\label{sec:semidirect-product-review-and-gen}

In a general setting, we first review the (inner) semidirect product of two subgroups, and then generalizes it to the semidirect product of $k$ subgroups.
The resulting $k$-ary semidirect-product decomposition of a group is the core structure upon which we build the main results of this paper.

\begin{definition}\label{def:product-of-sets}
Let $G$ be a group, and $A_1, \ldots, A_k \subseteq G$ be subsets.
We define the product of these subsets (which itself is a subset of $G$) by
\begin{align}\label{eqn:product-of-sets}
A_k\cdots A_1 := \{a_k\cdots a_1 \mid a_i \in A_i, i = 1, \ldots, k\}.
\end{align}
\end{definition}

\begin{definition}\label{def:semidirect-product-two-inner}
Let $G$ be a group, and $A,B \leq G$ be two subgroups.
We write the bracket notation $\left[AB\right]$ to mean:
\begin{enumerate}[label=$\langle\,\arabic*.\,\rangle$]
\item $\left[AB\right] = AB$ as a set;
\item $B\leq {\rm N}_G(A)$ where ${\rm N}_G(A)$ denotes the normalizer of $A$ in $G$;
\item $A \cap B = \{e\}$ where $e$ denotes the identity element in $G$.
\end{enumerate}
We call $\left[AB\right]$ the \emph{(inner) semidirect product} of $A$ and $B$.
(\underline{Note:} one can check that $[AB] \leq G$).
Further, if $G = \left[AB\right]$, then we say $G$ is the semidirect product of $A$ and $B$, or $[AB]$ is a semidirect-product decomposition of $G$.
\end{definition}

This definition is readily generalized to $k$ subgroups as follows.

\begin{definition}\label{def:semidirect-product-decomp}
Let $G$ be a group, and $A_1, \ldots, A_k \leq G$ be $k$ subgroups.
We define $\ginner$ recursively (on $k$) by
\begin{align}\label{eqn:semidirect-product-decomp}
\ginner := \left[A_k\left[A_{k-1}\cdots A_1\right]\right] \quad \mbox{ where for consistency } \left[ A_1 \right] := A_1.
\end{align}
Further, if $G = \ginner$, then we say $G$ is the $k$-ary semidirect product of $A_k, \ldots, A_1$, or $\ginner$ is a \emph{$k$-ary semidirect-product decomposition} of $G$.
%A \emph{$k$-ary semidirect-product decomposition} of a group $G$ is a right-nested sequence of semidirect products of $k$ subgroups of $G$, denoted by
%\begin{align}\label{eqn:semidirect-product-decomp}
%\gsdpd \quad \mbox{ for } k > 1,
%\end{align}
%where the bracket notation $\left[A_k \rtimes \cdots \rtimes A_1\right]$ is recursively defined as follows:
%\begin{align*}
%\left[A_k \rtimes \cdots \rtimes A_1\right] :=
%\begin{cases}
%A_2 \rtimes A_1 & k = 2 \\
%A_k \rtimes \left[A_{k-1} \rtimes \cdots \rtimes A_1\right] & k > 2
%\end{cases},
%\end{align*}
%or more explicitly, $\left[A_k \rtimes \cdots \rtimes A_1\right] := A_k \rtimes (\cdots \rtimes (A_{3} \rtimes (A_{2}\rtimes A_1)))$.
\end{definition}

\begin{remark}
Based on the binary bracket notation in Definition~\ref{def:semidirect-product-two-inner}, the following information is automatically encoded in the $k$-ary notation $\ginner$:
\begin{enumerate}[label=$\blacksquare~$, itemsep=-0.03in]
\item $\left[A_j \cdots A_1 \right] \leq G$ recursively for all $j \in \{2, \ldots, k\}$;
\item part of this notation is the requirement that $\left[A_{j-1}\cdots A_1\right] \leq {\rm N}_G(A_j)$ for all $j \in \{2, \ldots, k\}$, which further implies that $A_j \trianglelefteq \left[A_j\cdots A_1\right]$;
\item $A_j \cap \left[ A_{j-1}\cdots A_1\right] = \{e\}$ for all $j \in \{2, \ldots, k\}$, which further implies that $A_i \cap A_j = \{e\}$ for any distinct $i, j \in \{1, \ldots, k\}$.
\end{enumerate}
%For any semidirect-product decomposition of a group $G$ (if there exists one), it is necessary that the intersection of any two subgroups in the decomposition is the singleton $\{e\}$ containing only the identity element of $G$.
\end{remark}

\begin{remark}[Interesting comparison to subnormal series]
The notion of a $k$-ary semidirect-product decomposition is \emph{dual} to a subnormal series.
Compare $G = A_r \geq A_{r-1} \geq \cdots \geq A_{0} = 1$ with $G = \ginner$.
In the former, we have $A_i \trianglelefteq A_{i+1}$;
in the latter, we have $A_j \trianglelefteq \left[A_j\cdots A_1\right]$.
For a subnormal series, one keeps finding a normal subgroup within a normal subgroup;
whereas for a $k$-ary semidirect-product decomposition, one keeps finding a normal subgroup within the quotient by a normal subgroup.
Figure~\ref{fig:subnormal-series-and-semidirect-product} summarizes this duality.
\end{remark}

\begin{figure}[t!]
\begin{center}
\includegraphics[width=0.95\columnwidth]{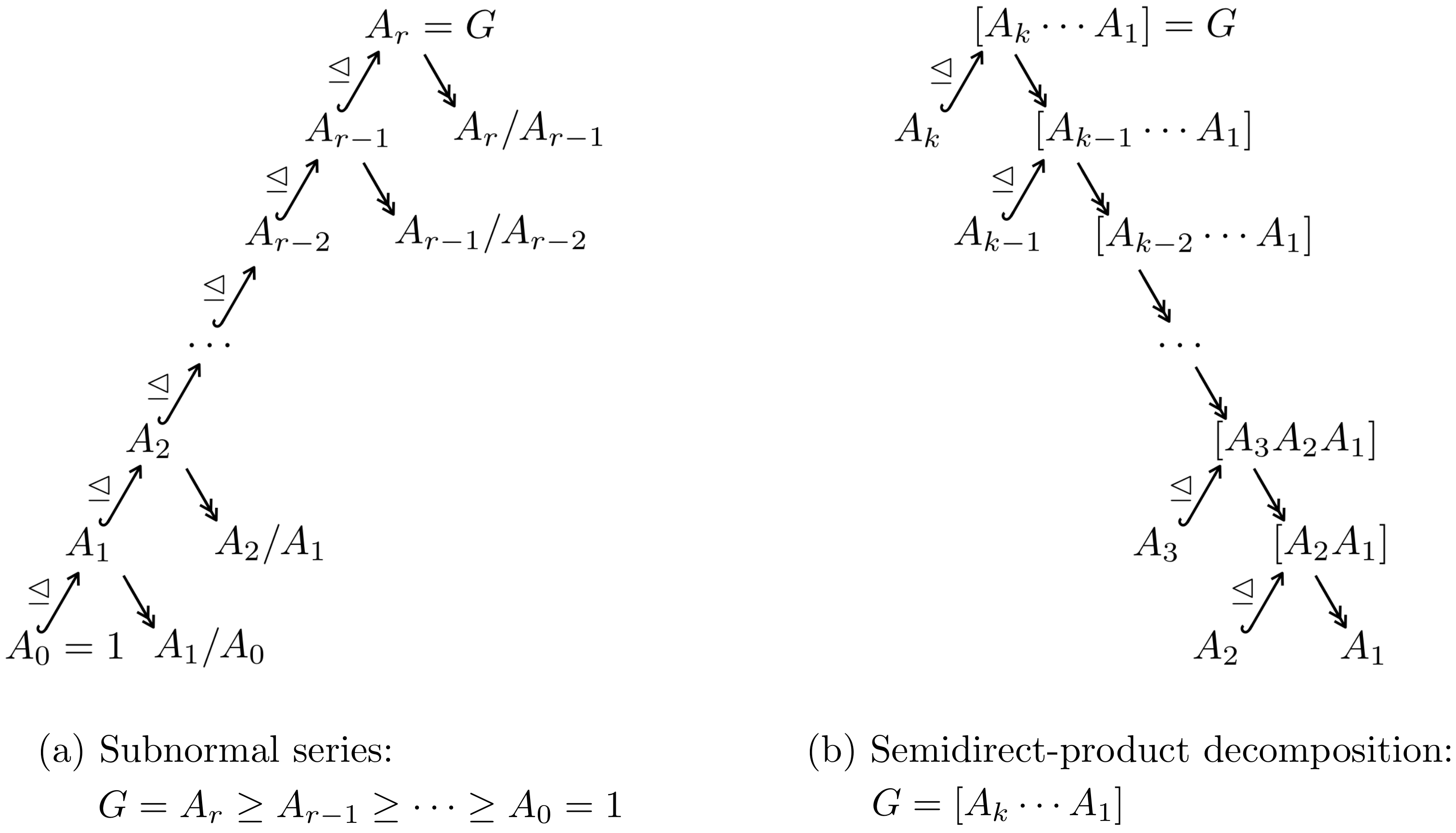}
\end{center}
\caption{Compare a subnormal series (left) with a semidirect-product decomposition (right).
For a subnormal series, we have a normal subgroup within a normal subgroup: $A_i \trianglelefteq A_{i+1}$;
for a semidirect-product decomposition, we have a normal subgroup within the complement of a normal subgroup: $A_j \trianglelefteq \left[A_j\cdots A_1\right]$.}
\label{fig:subnormal-series-and-semidirect-product}
\end{figure}

\section{The Mathematical Objects}
\label{sec:the-math-objects}

To prepare for the formal derivation of our orbit computation problem later (in Section~\ref{sec:the-problem-of-orbit-computation}), we first introduce the main mathematical object in this paper, namely \emph{the isometry group of $\ambsp$} denoted by $\iz$, and then a special class of subgroups of $\iz$, namely the class of \emph{atomically generated subgroups}.

\subsection{The Isometry Group of $\ambsp$: $\iz$}
\label{sec:the-isometry-group-of-zn}

Our ambient space is the metric space $(\ambsp, \dist)$, where $\dist: \ambsp \times \ambsp \to \reals$ is the Euclidean distance.
An \emph{isometry of $\ambsp$} is a function $h: \ambsp \to \ambsp$ satisfying the distance-preserving property: $\dist(h(x),h(x'))  = \dist(x,x')$ for any $x,x' \in \ambsp$.
We use $(\iz, \circ)$ to denote the group of all isometries of $\ambsp$ or in short, \emph{the isometry group of $\ambsp$}, which can be characterized via a semidirect product~\cite{YuMV2019}:
\begin{align}\label{eqn:isom-zn-char}
\iz = \left[ \tz \circ \rz \right].
\end{align}
In the above characterization:
\begin{enumerate}[label=$\blacksquare~$, itemsep=-0.03in]
\item $(\tz, \circ)$ denotes the group of translations of $\ambsp$, where a \emph{translation of $\ambsp$} is a function $t_v: \ambsp \to \ambsp$ defined by $t_v(x) := x + v$ with the parameter $v \in \ambsp$ being called the \emph{translation vector};
\item $(\rz, \circ)$ denotes the group of (generalized) rotations of $\ambsp$, where a \emph{rotation of $\ambsp$} is a function $r_R: \ambsp \to \ambsp$ defined by $r_R(x) := Rx$ with the parameter $R \in \onz$ being called the \emph{rotation matrix}.
\underline{Note:}
$\onz := \{R \in \integers^{n \times n} \mid R^\top = R^{-1}\}$; the word \emph{rotation} throughout this paper is a shorthand term for, more precisely, \emph{generalized rotation about the origin}, which is linear and can be either a proper rotation (whose rotation matrix has determinant $1$) or an improper rotation (whose rotation matrix has determinant $-1$).
\end{enumerate}

$\iz$ has an additional property regarding a finer dissection of $\rz$~\cite{YuMV2019}.
Expressed by another semidirect product at a smaller scale,
\begin{align}\label{eqn:rota-zn-char}
\rz = \left[ \nz \circ \pz \right].
\end{align}
In the above characterization:
\begin{enumerate}[label=$\blacksquare~$, itemsep=-0.03in]
\item $(\nz, \circ)$ denotes the group of (coordinate-wise) negations of $\ambsp$, where a \emph{negation of $\ambsp$} is a rotation $r_N: \ambsp \to \ambsp$ with the rotation matrix $N$ being a \emph{negation matrix}---a diagonal matrix whose diagonal entries are $\pm 1$.
\underline{Note:}
the word \emph{negation} throughout this paper is a shorthand term for, more precisely, \emph{coordinate-wise negation}, which negates some (or all) coordinates of a vector.
Here are two examples of a negation matrix:
\begin{align*}
N = \begin{bmatrix} -1 & 0 & 0 \\ 0 & 1 & 0 \\ 0 & 0 & -1\end{bmatrix}, \quad
N' = \begin{bmatrix} -1 & 0 & 0 \\ 0 & -1 & 0 \\ 0 & 0 & 1\end{bmatrix}.
\end{align*}
They induce two negations of $\integers^3$, $r_N$ and $r_{N'}$, such that for any $x = (x_1, x_2, x_3) \in \integers^3$, $r_N(x) = (-x_1, x_2, -x_3)$ and $r_{N'}(x) = (-x_1, -x_2, x_3)$.
\item $(\pz, \circ)$ denotes the group of (coordinate-wise) permutations of $\ambsp$, where a \emph{permutation of $\ambsp$} is a rotation $r_P: \ambsp \to \ambsp$ with the rotation matrix $P$ being a \emph{permutation matrix}---a matrix obtained by permuting the rows of an identity matrix.
\underline{Note:}
the word \emph{permutation} throughout this paper is a shorthand term for, more precisely, \emph{coordinate-wise permutation}, which permutes the coordinates of a vector.
Here are two examples of a permutation matrix:
\begin{align*}
P = \begin{bmatrix} 0& 1 & 0 \\ 1 & 0 & 0 \\ 0 & 0 & 1\end{bmatrix}, \quad
P' = \begin{bmatrix} 0 & 0 & 1 \\ 1 & 0 & 0 \\ 0 & 1 & 0\end{bmatrix}.
\end{align*}
They induce two permutations of $\integers^3$, $r_P$ and $r_{P'}$, such that for any $x = (x_1, x_2, x_3) \in \integers^3$, $r_P(x) = (x_2, x_1, x_3)$ and $r_{P'}(x) = (x_3, x_1, x_2)$.
\end{enumerate}
Clearly, the rotation group $\rz$ is finite: $|\rz| = |\nz|\cdot |\pz| = 2^n(n!)$.

Expressions~\eqref{eqn:isom-zn-char} and \eqref{eqn:rota-zn-char} reveal the semidirect-product structure at two different scales.
Putting them together, we have a ternary semidirect product:
\begin{align}\label{eqn:isom-zn-ternary-semidirect-product}
\iz = \left[ \tz \circ \nz \circ \pz \right].
\end{align}
%Therefore, any isometry $h \in \iz$ can be uniquely represented as:
%\begin{align}\label{eqn:isom-unique-representation}
%h(x) = t_v \circ r_N \circ r_P(x) = NPx + v \quad \mbox{ for all } x \in \ambsp,
%\end{align}
%where $v \in \ambsp$ is the translation vector, $N \in \integers^{n \times n}$ is the negation matrix, and $P \in \integers^{n \times n}$ is the permutation matrix.

\subsection{Atomically Generated Subgroups}
\label{sec:atomically-gen-subgroups}

In this paper, we consider a special class of subgroups of $\iz$, namely the class of \emph{atomically generated subgroups} of $\iz$.
Every atomically generated subgroup has a so-called \emph{atomic generating set}.
To formally introduce the notion of \emph{atomic}, we start with definitions in a more general setting.

\begin{definition}\label{def:atomic-set}
Let $G = \ginner$ be a semidirect-product decomposition of $G$.
A subset $S \subseteq G$ is \emph{atomic (with respect to the semidirect-product decomposition)}, if
\begin{align}\label{eqn:atomic-wrt-semidirect-product-decomp}
S \subseteq \atomicsuper.
\end{align}
\end{definition}

\begin{definition}\label{def:atomically-generated}
Let $G = \ginner$ be a semidirect-product decomposition of $G$.
A subgroup $H \leq G$ is \emph{atomically generated (with respect to the semidirect-product decomposition)}, if it has an atomic generating set, \ie there exists an atomic subset $S \subseteq G$ such that $H = \langle S \rangle$.
\end{definition}

Returning to our main mathematical object $\iz$, we have so far introduced two semidirect-product decompositions of it, namely the binary one in Expression~\eqref{eqn:isom-zn-char} and the ternary one in Expression~\eqref{eqn:isom-zn-ternary-semidirect-product}.
In the sequel, if the decomposition of $\iz$ is not explicitly specified, we assume it is by default the ternary decomposition in Expression~\eqref{eqn:isom-zn-ternary-semidirect-product}.
Therefore, a subset of isometries $S \subseteq \iz$ is atomic if $S \subseteq \tz \cup \nz \cup \pz$.

\vspace{0.1in}
Before closing the section, we introduce a shorthand notation for referencing any component of an atomic subset, as well as any component of an atomically generated subgroup.
It is designed to make such references simple, systematic, and consistent with the underlying semidirect-product decomposition.

Let $G = \ginner$ be a semidirect-product decomposition of $G$, and let $2^G$ denote the power set of $G$.
For every $i \in \{1,\ldots, k\}$, define the function $\acal_i: 2^G \to 2^G$ and its subscript shorthand notation by
\begin{align}\label{eqn:shorthand-notation-for-semidirect-product-components}
(\acal_i)_S := S \cap A_i \quad \mbox{ for any } S \subseteq G.
\end{align}
Note that the above notation and definition apply to all subsets of $G$.
However, their main use will be for atomic subsets and atomically generated subgroups.
First, for any atomic subset $S \subseteq G$, it is immediate from Definition~\ref{def:atomic-set} that $S$ can be always decomposed as follows:
\begin{align}\label{eqn:atomic-set-decomp}
S = (\acal_1)_S \cup \cdots \cup (\acal_k)_S.
\end{align}
Indeed, one can check that Equation~\eqref{eqn:atomic-set-decomp} holds if and only if $S$ is atomic.
Second, for any atomically generated subgroup $H \leq G$, we will soon see (in Section~\ref{sec:specialty-of-atomically-generated-subgroups}: Theorem~\ref{thm:distinguishing-property-atomically-generated-subgroup-k-ary}) that $H$ can be always decomposed as follows:
\begin{align}\label{eqn:atomically-generated-subgroup-decomp}
H = \sginner{H},
\end{align}
inheriting the semidirect-product structure from $G$.
%where the outer semidirect-product decomposition $\sgouter{H}$ of $H$ inherits the same parentheses-nesting pattern from the outer semidirect-product decomposition $\left[A_1 \rtimes \cdots \rtimes A_k\right]$ of $G$.
A sanity check: when $H = G$, Expression~\eqref{eqn:atomically-generated-subgroup-decomp} is precisely $G = \ginner$.

In our special case when $G = \iz$, we have the following four particular notations: for any $S \subseteq \iz$,
\begin{align*}
\tcal_S := S \cap \tz,~~
\rcal_S := S \cap \rz,~~
\ncal_S := S \cap \nz,~~
\pcal_S := S \cap \pz,~~
\end{align*}
which represent the (pure) translations, rotations, negations, and permutations in $S$, respectively.
Further, for any atomic $S \subseteq \iz$,
\begin{align*}
S &= \tcal_S \cup \rcal_S = \tcal_S \cup \ncal_S \cup \pcal_S, \\
\langle S \rangle &= \left[\tcal_{\langle S \rangle} \circ \rcal_{\langle S \rangle}\right] = \left[\tcal_{\langle S \rangle} \circ \ncal_{\langle S \rangle} \circ \pcal_{\langle S \rangle}\right].
\end{align*}
Again, the second line in the above will be clear after the following section.

\section{Special Property of Atomically Generated Subgroups}
\label{sec:specialty-of-atomically-generated-subgroups}

We describe a special property of atomically generated subgroups.
As every $k$-ary semidirect-product decomposition is recursively built from binary ones, it suffices to focus on binary semidirect-product decomposition.
We start from groups with a binary semidirect-product decomposition in general, then generalize to $k$-ary decompositions, and finally apply the results to isometries.

Let $G = \left[AB\right]$ be a binary semidirect-product decomposition of a group $G$, then for any $g \in G$, there exists a unique $a \in A$ and a unique $b \in B$ such that $g = ab$.
This uniqueness allows us to define a function $\varphi_A: G \to A$ and a function $\varphi_B: G \to B$, such that for any $g \in G$, $g = \varphi_A(g)\varphi_B(g)$.
The main task in this section is: for any atomic subset $S \subseteq G$, characterize three subgroups namely $\langle S \rangle$, $\acal_{\langle S \rangle}$, $\bcal_{\langle S \rangle}$.
We first present the conclusion as Theorem~\ref{thm:distinguishing-property-atomically-generated-subgroup-2-ary}.

\vspace{0.1in}\noindent
\fbox{\parbox{\textwidth}{
\begin{theorem}[\bf Special Property of Atomically Generated Subgroup: Binary Case]
\label{thm:distinguishing-property-atomically-generated-subgroup-2-ary}
Let $G = \left[AB\right]$, and $S \subseteq G$ be atomic, then
\begin{gather*}
\langle S \rangle = \left[\acal_{\langle S \rangle} \bcal_{\langle S \rangle}\right] \quad \mbox{where} \\
\acal_{\langle S \rangle} := \langle S \rangle \cap A = \langle \acal_S^+ \rangle = \varphi_A(\langle S \rangle), \\
\bcal_{\langle S \rangle} := \langle S \rangle \cap B  = \langle \bcal_S \rangle = \varphi_B(\langle S \rangle).
\end{gather*}
$\acal_S^+ := \acal_S^{B_{\langle S \rangle}} := \{(a')^b\mid a' \in \acal_S, b \in \bcal_{\langle S \rangle}\}$ (where $(a')^b:=ba'b^{-1}$) is called the \emph{augmented generating set}: augmented from $\acal_S$ through conjugation by $\bcal_{\langle S \rangle}$.
Hence, $\acal_{\langle S \rangle}$ is the conjugate closure (or normal closure) of $\acal_S$ under $\bcal_{\langle S \rangle}$.
\end{theorem}
}}
\vspace{0.1in}

\noindent To prove Theorem~\ref{thm:distinguishing-property-atomically-generated-subgroup-2-ary}, we break it down into Theorems~\ref{thm:char-pure-a-in-s-gen-via-generating-set}--\ref{thm:semi-direct-product-for-decomposable-s}.
Further, for Theorems~\ref{thm:char-pure-a-in-s-gen-via-generating-set}--\ref{thm:char-pure-b-in-s-gen-via-phi-b}, we first state a result without proof: for any $g \in \langle S \rangle$, $\varphi_A(g) \in \langle \acal_S^+ \rangle$ and $\varphi_B(g) \in \langle \bcal_S \rangle$; or equivalently, $\varphi_A(\langle S \rangle) \subseteq \langle \acal_S^+ \rangle$ and $\varphi_B(\langle S \rangle) \subseteq \langle \bcal_S \rangle$.

%\begin{lemma}\label{lemma:image-phi-a-is-subset-generated-set}
%$\varphi_A(\langle S \rangle) \subseteq \langle \acal_S^+ \rangle$.
%\end{lemma}
%\begin{lemma}\label{lemma:image-phi-b-is-subset-generated-set}
%$\varphi_B(\langle S \rangle) \subseteq \langle \bcal_S \rangle$.
%\end{lemma}

\begin{theorem}\label{thm:char-pure-a-in-s-gen-via-generating-set}
$\acal_{\langle S \rangle} = \langle \acal_S^+ \rangle$.
\end{theorem}
\begin{proof}
For any $g \in \acal_{\langle S \rangle} = \langle S \rangle \cap A$, $g = \varphi_A(g) \in \langle \acal_S^+ \rangle$.
Thus, $\acal_{\langle S \rangle} \subseteq \langle \acal_S^{+} \rangle$.
Conversely, for any $g \in \langle \acal_S^+ \rangle$,
$g = (a'_k)^{b_k} \cdots (a'_1)^{b_1}$
for some $b_k, \ldots, b_1 \in \langle \bcal_S \rangle$, $a'_k, \ldots, a'_1 \in \acal_S$.
For any $j \in \{1, \ldots, k\}$, clearly, $(a'_j)^{b_j} \in \langle S \rangle$; further, $(a'_j)^{b_j} \in A$ since $A \trianglelefteq G$.
This implies that $g \in \langle S \rangle \cap A = \acal_{\langle S \rangle}$.
Thus, $\langle \acal_S^+ \rangle \subseteq \acal_{\langle S \rangle}$.
\end{proof}

\begin{theorem}\label{thm:char-pure-a-in-s-gen-via-phi-a}
$\acal_{\langle S \rangle}= \varphi_A(\langle S \rangle)$.
\end{theorem}
\begin{proof}
For any $g \in \acal_{\langle S \rangle} = \langle S \rangle \cap A$, $g = \varphi_A(g) \in \varphi_A(\langle S \rangle)$.
So, $\acal_{\langle S \rangle} \subseteq \varphi_A(\langle S \rangle)$.
Conversely, $\varphi_A(\langle S \rangle) \subseteq \langle \acal_S^+\rangle = \acal_{\langle S \rangle}$ (just proved in Theorem~\ref{thm:char-pure-a-in-s-gen-via-generating-set}).
\end{proof}

%\begin{remark}\label{remark:char-a-part}
%So far, we have done Part 1 of the plan:
%$\acal_{\langle S \rangle} = \langle \acal_S^+ \rangle = \varphi_A(\langle S \rangle)$.
%\end{remark}

\begin{theorem}\label{thm:char-pure-b-in-s-gen-via-generating-set}
$\bcal_{\langle S \rangle} = \langle \bcal_S \rangle$.
\end{theorem}
\begin{proof}
For any $g \in \bcal_{\langle S \rangle} = \langle S \rangle \cap B$, $g = \varphi_B(g) \in \langle \bcal_S \rangle$.
Thus, $\bcal_{\langle S \rangle} \subseteq \langle \bcal_S \rangle$.
Conversely, $\langle \bcal_S \rangle \leq \langle S \rangle$ and $\langle \bcal_S \rangle \leq \langle B \rangle = B$, thus, $\langle \bcal_S \rangle \subseteq \langle S \rangle \cap B = \bcal_{\langle S \rangle}$.
\end{proof}

\begin{theorem}\label{thm:char-pure-b-in-s-gen-via-phi-b}
$\bcal_{\langle S \rangle} = \varphi_B(\langle S \rangle)$.
\end{theorem}
\begin{proof}
For any $g \in \bcal_{\langle S \rangle} = \langle S \rangle \cap B$, $g = \varphi_B(g) \in \varphi_B(\langle S \rangle)$.
So, $\bcal_{\langle S \rangle} \subseteq \varphi_B(\langle S \rangle)$.
Conversely, $\varphi_B(\langle S \rangle) \subseteq \langle \bcal_S \rangle = \bcal_{\langle S \rangle}$ (just proved in Theorem~\ref{thm:char-pure-b-in-s-gen-via-generating-set}).
\end{proof}

%\begin{remark}\label{remark:char-b-part}
%So far, we have done Part 2 of the plan: $\bcal_{\langle S \rangle} = \langle \bcal_S \rangle = \varphi_B(\langle S \rangle)$.
%\end{remark}

%\vspace{0.1in}
%We proceed to the third part of the plan to characterize $\langle S \rangle$ via $\acal_{\langle S \rangle}$ and $\bcal_{\langle S \rangle}$.
%It remains to show that $\langle S \rangle = \left[\acal_{\langle S \rangle} \bcal_{\langle S \rangle}\right]$.
%We start with a more general setting: let $H \leq G = \left[AB\right]$ be \emph{any} subgroup of $G$, and consider the following two objects $\varphi_B(H)$ and $\acal_{H} := H \cap A$.
%Below are a few facts about $\varphi_B(H)$ and $\acal_{H}$, linked via the restriction $\varphi_B|_{H}$.
%\begin{enumerate}[label=$\blacksquare~$]
%\item $\varphi_B|_{H}$ is a homomorphism (as the restriction of the homomorphism $\varphi_B$).
%\item The image ${\rm Im}(\varphi_B|_{H}) = \varphi_B(H) \leq B$.
%\item The kernel ${\rm Ker}(\varphi_B|_{H}) = \acal_{H} \trianglelefteq H$.
%\item $H/\acal_{H} \cong \varphi_B(H)$ (by the first group isomorphism theorem).
%\end{enumerate}
%The above properties hold for all $H \leq G$ and particularly, for atomically generated $H = \langle S \rangle$ for some atomic subset $S = \acal_S \cup \bcal_S$.
%However, we can say more about atomically generated subgroups.

\begin{theorem}\label{thm:semi-direct-product-for-decomposable-s}
For any atomically generated subgroup $H \leq G$, $H = \left[\acal_{H} \bcal_{H}\right]$.
\end{theorem}
\begin{proof}
It is clear that $\acal_{H}, \bcal_{H} \leq H$ and $\acal_{H} \cap \bcal_{H} = \{e\}$, since by definition $\acal_{H} = H \cap A$, $\bcal_{H} = H \cap B$, and $A \cap B = \{e\}$;
further, it is clear that $\acal_{H} \trianglelefteq H$.
It remains to be shown that $H = \acal_{H} \bcal_{H}$.
Since $\acal_{H}, \bcal_{H} \leq H$, $\acal_{H} \bcal_{H} \subseteq H$.
Conversely, for any $g \in H \leq G = \left[AB\right]$, $g = \varphi_A(g)\varphi_B(g) \in \varphi_A(H)\varphi_B(H) = \acal_{H} \bcal_{H}$; thus, $H \subseteq \acal_{H} \bcal_{H}$.
Then, by Definition~\ref{def:semidirect-product-two-inner}, $H = \left[\acal_{H} \bcal_{H}\right]$.
\end{proof}
\begin{remark}
It is important that the subgroup $H$ in Theorem~\ref{thm:semi-direct-product-for-decomposable-s} is atomically generated, which guarantees that $\varphi_B(H)$ is also a subgroup of $H$ (Theorem~\ref{thm:char-pure-b-in-s-gen-via-phi-b}).
This is \emph{not} true for any $H \leq G$: for some $H \leq G$, $\varphi_B(H)$ is not even a subset of $H$.
For example, let $\mathsf{AFF}(\reals)$ denote the group of (invertible) affine transformations of $\reals$, $\tran(\reals)$ denote the group of translations of $\reals$, and $\mathsf{L}(\reals)$ denote the group of (invertible) linear transformations of $\reals$; further, for any $a,b \in \reals$, let $f_{(a,b)}: \reals \to \reals$ denote the affine transformation defined by $f_{(a,b)}(x) := ax + b$. Consider $G = \mathsf{AFF}(\reals) = \left[ {\rm T}(\reals) \circ {\rm L}(\reals) \right]$ and $H = \langle f_{(2,1)} \rangle = \{f_{(2^n, 2^n-1)}\mid n \in \integers\}$.
In this case, $B = {\rm L}(\reals)$ and $\varphi_B(H) = \{f_{(2^n,0)}\mid n \in \integers\} \not\subseteq H$.
\end{remark}

So far, we have proved Theorem~\ref{thm:distinguishing-property-atomically-generated-subgroup-2-ary}.
We can further generalize it to any $k$-ary semidirect-product decomposition which is stated as Theorem~\ref{thm:distinguishing-property-atomically-generated-subgroup-k-ary}.
The proof is done by induction and is relegated to Appendix~\ref{app:distinguishing-property-atomically-generated-subgroup-k-ary}.

\vspace{0.1in}\noindent
\fbox{\parbox{\textwidth}{
\begin{theorem}[\bf Special Property of Atomically Generated Subgroup: General Case]
\label{thm:distinguishing-property-atomically-generated-subgroup-k-ary}
Let $G = \ginner$, and $S\subseteq G$ be atomic.
Then, $\langle S \rangle$ has a similar semidirect-product decomposition:
\begin{gather}
\label{eqn:semidirect-product-decomp-for-atomically-generated-subgroup-k-ary}
\langle S \rangle = \sginner{\langle S \rangle}, \quad \mbox{ where } \\
\label{eqn:component-for-atomically-generated-subgroup-k-ary}
(\acal_j)_{\langle S \rangle} := \langle S \rangle \cap A_j = \langle (\acal_j)_S^+ \rangle = \varphi_{A_j}(\langle S \rangle) \quad \mbox{ for any } j \in \{1, \ldots, k\}.
\end{gather}
In Equation~\eqref{eqn:component-for-atomically-generated-subgroup-k-ary}, the augmented generating set is consistently defined as follows:
\begin{align}
\label{eqn:gen-set-plus-notation}
(\acal_j)_S^+ := ((\acal_j)_S)^{(\acal_{j-1})_{\langle S \rangle} \cdots (\acal_1)_{\langle S \rangle}}.
\end{align}
In particular, $(\acal_1)_S^+ = ((\acal_1)_S)^{\{e\}} = (\acal_1)_S$.
\end{theorem}
}}

\vspace{0.1in}
\paragraph{Special Property of Atomically Generated Subgroup in the Case of Isometries}~

Consider our main mathematical object: $\iz = \left[ \tz \circ \rz \right]$ where $\rz = \left[ \nz \circ \pz \right]$, or collectively, $\iz =  \left[ \tz \circ \nz \circ \pz \right]$.
For any $h \in \iz$, $h = t_v \circ r_R = t_v \circ r_N \circ r_P$ for some unique $t_v \in \tz$, $r_R \in \rz$, $r_N \in \nz$, and $r_P \in \pz$.
The uniqueness allows us to define $\varphi_{\tran}: \iz \to \tz$, $\varphi_{\rota}: \iz \to \rz$, $\varphi_{\nega}: \iz \to \nz$, and $\varphi_{\perm}: \iz \to \pz$, such that $h = \varphi_{\tran}(h) \circ \varphi_{\rota}(h) = \varphi_{\tran}(h) \circ \varphi_{\nega}(h) \circ \varphi_{\perm}(h)$.

We apply Theorem~\ref{thm:distinguishing-property-atomically-generated-subgroup-k-ary} to the above semidirect-product decompositions of $\iz$ and its subgroups.
For any atomic subset $S \subseteq \iz$,
\begin{align}
\label{eqn:atomically-generated-subgroup-isom-zn-binary-semidirect-product-decomp}
\langle S \rangle &= \left[ \tcal_{\langle S \rangle} \circ \rcal_{\langle S \rangle} \right], \\
\label{eqn:atomically-generated-subgroup-rota-zn-binary-semidirect-product-decomp}
\rcal_{\langle S \rangle} &= \left[ \ncal_{\langle S \rangle} \circ \pcal_{\langle S \rangle} \right], \quad \mbox{ or collectively, } \\
\label{eqn:atomically-generated-subgroup-isom-zn-ternary-semidirect-product-decomp}
\langle S \rangle &= \left[ \tcal_{\langle S \rangle} \circ \ncal_{\langle S \rangle} \circ \pcal_{\langle S \rangle} \right]; \\
\label{eqn:atomically-generated-subgroup-tran-zn-char}
\tcal_{\langle S \rangle} &= \langle \tcal_S^+ \rangle = \varphi_{\tran}(\langle S \rangle), \\
\label{eqn:atomically-generated-subgroup-rota-zn-char}
\rcal_{\langle S \rangle} &= \langle \rcal_S \rangle = \varphi_{\rota}(\langle S \rangle), \\
\label{eqn:atomically-generated-subgroup-nega-zn-char}
\ncal_{\langle S \rangle} &= \langle \ncal_{S}^+ \rangle = \varphi_{\nega}(\langle S \rangle), \\
\label{eqn:atomically-generated-subgroup-perm-zn-char}
\pcal_{\langle S \rangle} &= \langle \pcal_{S} \rangle = \varphi_{\perm}(\langle S \rangle).
\end{align}
In Equation~\eqref{eqn:atomically-generated-subgroup-tran-zn-char}, $\tcal_S^+ := (\tcal_S)^{\rcal_{\langle S \rangle}}$; in Equation~\eqref{eqn:atomically-generated-subgroup-nega-zn-char}, $\ncal_S^+ := (\ncal_S)^{\pcal_{\langle S \rangle}}$.

\section{The Problem of Orbit Computation}
\label{sec:the-problem-of-orbit-computation}

Having introduced our main mathematical object $\iz$ and its atomically generated subgroups, we now formally introduce our orbit computation problem.

Under the ambient group action $\iz \curvearrowright \ambsp$ (naturally defined by $h\cdot x := h(x)$ for any $h \in \iz$ and any $x \in \ambsp$),
we consider the induced subgroup group action $H \curvearrowright \ambsp$ for any given $H \leq \iz$, and want to compute the set of orbits $\ambsp/H := \{\hspace{0.01in}H\cdot x \mid x \in \ambsp\hspace{0.01in}\}.$
This problem is only mathematically but not computationally well-defined, since the subgroup $H$ can be infinite and the desired output $\ambsp/H$ may comprise a possibly infinite number of orbits of possibly infinite size.
To make the problem practical, we must clarify: how we represent $H$ if it is infinite; how we represent $\ambsp/H$, \ie a partition of an infinite space.
For the former, we focus on subgroups represented by a finite generating set (in fact, one can check that all subgroups of $\iz$ are finitely generated);
for the latter, by computing a partition of an infinite space we mean being able to compute this partition restricted to \emph{any} finite subset $Z \subseteq \ambsp$.
This yields the following computationally well-defined problem:
\begin{equation}
\label{eqn:orbit-computation-practical}
\begin{aligned}
\mbox{\textbf{Inputs:}} \quad & \mbox{1) any finite } S \subseteq \iz; \\
& \mbox{2) any finite } Z \subseteq \ambsp. \\
\mbox{\textbf{Output:}} \quad & Z/H := (\ambsp/H)|_Z := \{\hspace{0.03in}(H\cdot x) \cap Z \mid x \in Z\hspace{0.03in}\}.
\end{aligned}
\end{equation}
%Note the partition $Z/H$ is not a standard notation; in particular, the group action considered here is still $H\curvearrowright \ambsp$ (not the undefined $H \curvearrowright Z$).
%However, the above definition of $Z/H$ is more general ($Z/H$ is indeed $\ambsp/H$ when $Z = \ambsp$), and should be interpreted as the set of orbits under $H\curvearrowright \ambsp$ in the scope of $Z$.

Aiming for a high efficiency that can be achieved via a full exploitation of semidirect products, in this paper, we focus on subgroups with a finite and atomic generating set, since all these subgroups preserve the semidirect-product structure from $\iz$ (Section~\ref{sec:specialty-of-atomically-generated-subgroups}).
This finally yields Problem~\eqref{eqn:orbit-computation-practical-special} as follows.

\vspace{0.1in}\noindent
\fbox{\parbox{\textwidth}{
\vspace{0.1in}\centerline{\bf Our Restricted Orbit Computation Problem}\vspace{-0.1in}
\begin{equation}
\label{eqn:orbit-computation-practical-special}
\begin{aligned}
\mbox{\textbf{Inputs:}} \quad & \mbox{1) any finite and atomic } S \subseteq \iz; \\
& \mbox{2) any finite } Z \subseteq \ambsp. \\
\mbox{\textbf{Output:}} \quad & Z/\langle S \rangle := \{\hspace{0.03in}(\langle S \rangle\cdot x) \cap Z \mid x \in Z\hspace{0.03in}\}.
\end{aligned}
\end{equation}
}}

\vspace{0.2in}
The desired output $Z/\langle S \rangle$ is a partition of the finite subset $Z$ which we represent explicitly via a labeling function $\lambda$.
A desired labeling function assigns every $x \in Z$ a label $\lambda(x)$ that indicates the orbit $(\langle S \rangle \cdot x)\cap Z$ (a set).
As part of a design choice, there is much freedom in picking orbit labels, either numerical or categorical, to replace the original set representation.

As a result, solving the orbit computation problem boils down to designing an algorithm that implements such a labeling function.
In this paper, we always use a point $\omega \in \ambsp$ to label an orbit of points.
More details on orbit labeling will be discussed in the next section, which completes the computational formulation of our restricted orbit computation problem before proceeding to the algorithm.

\section{Orbit Labeling}
\label{sec:orbit-labeling}

First in a most general sense, we introduce the notion of an \emph{orbit-labeling map} and its connection to the \emph{orbit-quotient map}; then we introduce the notion of an \emph{orbit-representative map}---a special type of orbit-labeling map.
In this section, we let $G$ be a group, and $X$ be a set whose elements are called \emph{points}.

\begin{definition}\label{def:orbit-labeling-map}
Let $G\curvearrowright X$ be a $G$-action on $X$, and $Z \subseteq X$ be a subset.
We call a function $\lambda: Z \to L$ an \emph{orbit-labeling map of $G \curvearrowright X$ on $Z$} if it satisfies the following \emph{orbit-labeling property} (on $Z$):
\begin{align}\label{eqn:orbit-labeling-property}
\lambda(x) = \lambda(x') \iff G\cdot x = G\cdot x' \quad \mbox{ for any } x, x' \in Z.
\end{align}
\end{definition}
\begin{remark}
In words, an orbit-labeling map assigns every point in $Z$ a label in $L$, and two points get the same label if and only if they are in the same orbit.
Therefore, collecting the (non-empty) preimages of the labels precisely recovers the set of the orbits in the scope of $Z$, \ie $Z/G = \{\hspace{0.01in}(G\cdot x) \cap Z \mid x \in Z\hspace{0.01in}\}$.
It is clear that a orbit-labeling map is essentially the orbit-quotient map (Theorem~\ref{thm:orbit-labeling-map}).
\end{remark}

\begin{theorem}\label{thm:orbit-labeling-map}
Let $G\curvearrowright X$ be a $G$-action on $X$, and $q: X \to X/G$ be the \emph{orbit-quotient map} \ie $q(x) := G\cdot x$.
Further, for any $Z \subseteq X$, we introduce the \emph{strong restriction of $q$ on $Z$} to be the surjective function $q||_Z: Z \to q(Z)$ defined by $q||_Z(x) := q(x)$.
Then, a function $\lambda: Z \to L$ is an orbit-labeling map of $G\curvearrowright X$ on $Z$ if and only if there exists an injective function $\overline{\lambda}: q(Z) \to L$ such that the diagram in Figure~\ref{fig:orbit-labeling-map} is commutative: $\lambda = \overline{\lambda} \circ {q}||_Z$.
\begin{figure}[h!]
\begin{center}
\includegraphics[width=0.25\columnwidth]{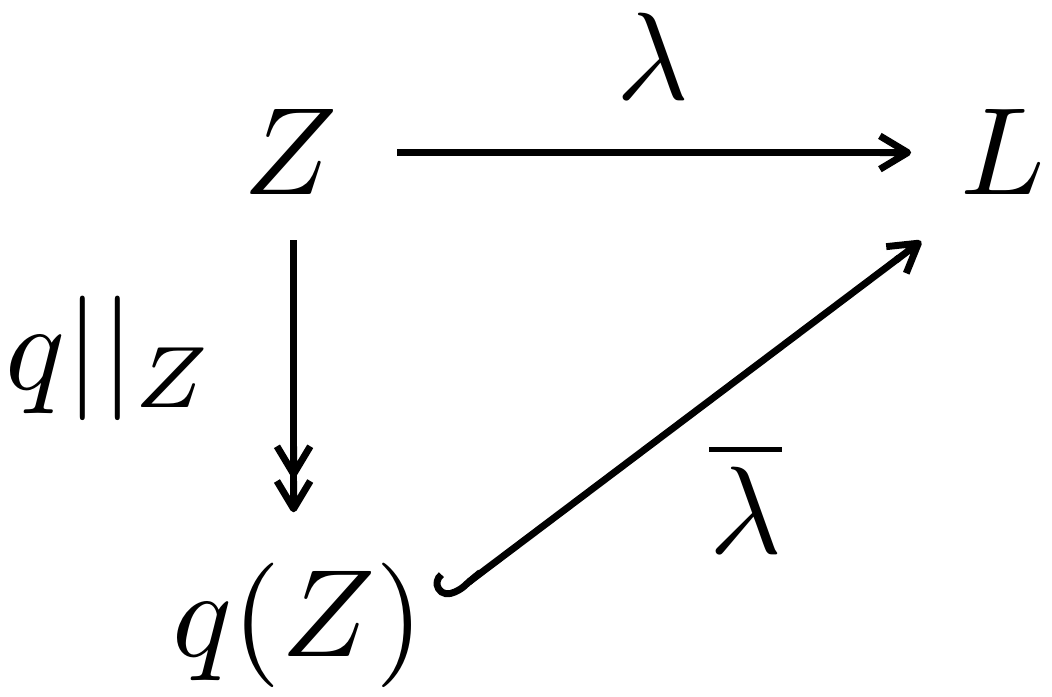}
\end{center}
\caption{A commutative diagram illustrating the connection between an orbit-labeling map ($\lambda$) and the orbit-quotient map ($q$): $\lambda = \overline{\lambda} \circ q||_Z$.}
\label{fig:orbit-labeling-map}
\vspace{-0.0005in}
\end{figure}
\end{theorem}

It does not matter what the labels are in the orbit-labeling property~\eqref{eqn:orbit-labeling-property}, so we have full freedom in both designing the labels ($L$) and assigning them to points in the same orbits.
Yet in this paper, we focus on labels that are points ($L = X$), and if possible make them representatives of the orbits (Definition~\ref{def:orbit-representative-map}).
%This brings the notion of an \emph{orbit-representative map}---a special orbit-labeling map.

\begin{definition}\label{def:orbit-representative-map}
Let $G\curvearrowright X$ be a $G$-action on $X$, and $q: X \to X/G$ be the orbit-quotient map.
We call a function $\rho: X \to X$ an \emph{orbit-representative map of $G \curvearrowright X$} if there exists a section $\overline{\rho}: X/G \to X$ of $q$ (\ie a right inverse of $q$: $q \circ \overline{\rho} = \idty_{_{X/G}}$) such that the diagram in Figure~\ref{fig:orbit-representative-map} is commutative: $\rho = \overline{\rho} \circ q$.
\begin{figure}[h!]
\begin{center}
\includegraphics[width=0.24\columnwidth]{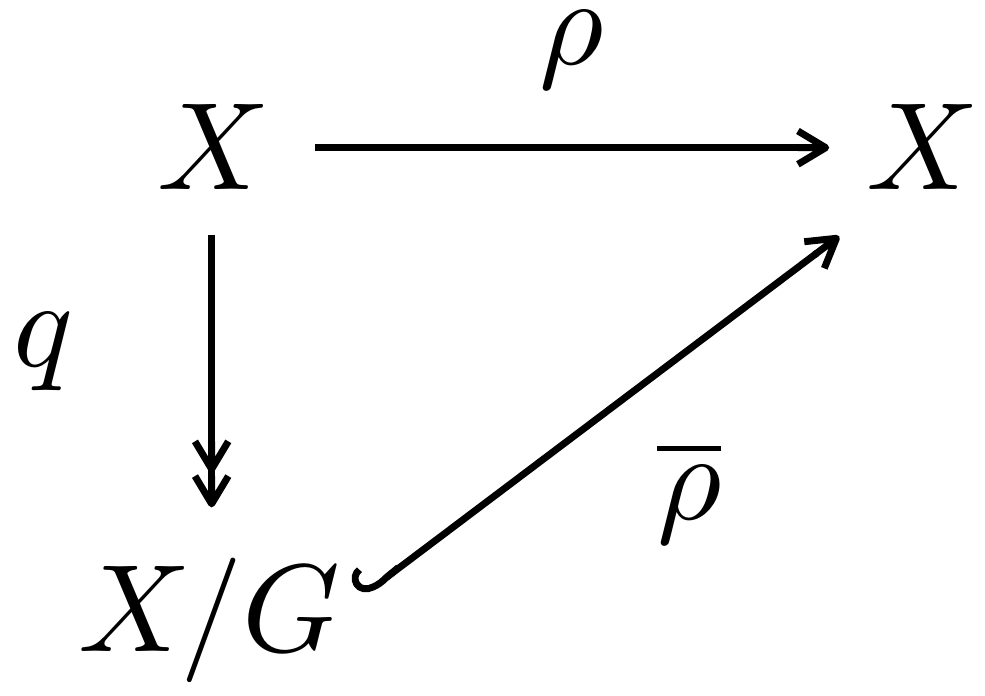}
\end{center}
\caption{A commutative diagram illustrating an orbit-representative map ($\rho$). In this diagram, $\overline{\rho} \circ q = \rho$; oppositely, $q \circ \overline{\rho} = \idty_{_{X/G}}$.}
\label{fig:orbit-representative-map}
\end{figure}
\end{definition}
\begin{remark}\label{remark:orbit-representative-map-is-orbit-labeling-map}
Comparing the two diagrams in Figures~\ref{fig:orbit-labeling-map} and \ref{fig:orbit-representative-map}, we can see that any orbit-representative map of $G \curvearrowright X$ is automatically an orbit-labeling map of $G \curvearrowright X$ on $X$, by making the following special choices in Figure~\ref{fig:orbit-labeling-map}:
\begin{enumerate}[label=$\blacksquare~$, itemsep=-0.03in]
\item $Z = X$, $L = X$, and further,
\item $\overline{\lambda}$ is not only injective but also a section of $q$.
\end{enumerate}
\end{remark}
%\begin{remark}
%For any two sections $\lambda, \lambda'$ of $\qmap{X/G}$,
%\begin{align*}
%\lambda \circ \qmap{X/G} = \lambda' \circ \qmap{X/G} ~\implies~ \lambda \circ \qmap{X/G} \circ \lambda = \lambda' \circ \qmap{X/G} \circ \lambda ~\implies~ \lambda = \lambda'.
%\end{align*}
%So, the notation $\overline{\lambda}$ is unambiguous, \ie given any representative quotient map, there is a \emph{unique} section $\lambda$ such that $ \overline{\lambda} = \lambda \circ \qmap{X/G}$ is the given map.
%\end{remark}

\begin{theorem}\label{thm:three-properties-of-rep-qmap}
Let $\rho: X \to X$ be an orbit-representative map of a group action $G \curvearrowright X$, then the following three properties hold:
\begin{enumerate}[label=$\langle\,{\rm P}\arabic*.\, \rangle$, itemindent=0.15in, itemsep=-0.03in]
\item\label{item1:three-properties-of-rep-qmap}
\emph{orbit-labeling:} for any $x,x' \in X$, $\rho(x) = \rho(x') \iff G\cdot x = G\cdot x'$.
\item\label{item2:three-properties-of-rep-qmap}
\emph{representative:} for any $x \in X$, $G \cdot \rho(x) = G\cdot x$.
\item\label{item3:three-properties-of-rep-qmap}
\emph{idempotent:} $\rho \circ \rho = \rho$.
\end{enumerate}
\end{theorem}
\begin{proof}
All properties are immediate from the definition of an orbit-representative map $\rho = \overline{\rho} \circ q$ where $\overline{\rho}$ is a section of the orbit-quotient map $q$ (of $G \curvearrowright X$).
\begin{enumerate}[label=$\langle\,{\rm P}\arabic*.\, \rangle$, itemindent=0.15in, topsep=0.03in, itemsep=-0.03in]
\item The orbit-labeling property is immediate from Remark~\ref{remark:orbit-representative-map-is-orbit-labeling-map}.
\item For any $x \in X$, $G \cdot \rho(x) = q \circ (\overline{\rho} \circ q)(x) = (q \circ \overline{\rho}) \circ q(x) = q(x) = G\cdot x$.
\item $\rho \circ \rho = (\overline{\rho} \circ q) \circ (\overline{\rho} \circ q) = \overline{\rho} \circ (q \circ \overline{\rho}) \circ q = \overline{\rho} \circ q = \rho$.
\vspace{-0.24in}
\end{enumerate}
\end{proof}

\newpage
\begin{remark}\label{remark:three-properties-of-rep-qmap}
We interpret Theorem~\ref{thm:three-properties-of-rep-qmap} in words as follows:
\begin{enumerate}[label=$\blacksquare~$, itemsep=-0.03in]
\item \ref{item1:three-properties-of-rep-qmap} and \ref{item2:three-properties-of-rep-qmap} collectively indicate that ${\rm Im}\rho = \rho(X)$ is a \emph{fundamental domain} of $G \curvearrowright X$ in the sense that it is a subset of $X$ containing exactly one point---called a representative---from each of the orbits;
\item \ref{item3:three-properties-of-rep-qmap} indicates that $\rho$ is a \emph{projection} onto the fundamental domain ${\rm Im}\rho$ in the sense that it is an idempotent.
\end{enumerate}
Therefore, $\rho$ projects every $x \in X$ to a chosen \emph{representative} of the orbit of $x$.
The representative $\rho(x) = \overline{\rho}(G\cdot x)$ is precisely designated by the section $\overline{\rho}$; the fundamental domain ${\rm Im}\rho = {\rm Im}\overline{\rho}$ is the set of representatives, which bijectively corresponds to the quotient $X/G$.
%This explains why we call $\rho$ an \emph{orbit-representative map}: it uses representatives to label orbits in the quotient.
\end{remark}

Conversely, one can check: none of the three properties alone is sufficient to make a function an orbit-representative map.
%This insufficiency is obvious for \ref{item2:three-properties-of-rep-qmap} and \ref{item3:three-properties-of-rep-qmap}, \eg considering $\rho = \idty_X$.
%To see \ref{item1:three-properties-of-rep-qmap} alone is also insufficient, we give an example.
%Let $G = 2\integers$, $X = \integers$, and the group action $G \curvearrowright X$ be defined by $g\cdot x := g+x$ for any $g \in G$ and any $x \in X$, then it is clear that in this case $X/G = \{2\integers, 2\integers+1\}$, where the two orbits are the set of even numbers and the set of odd numbers.
%Let $\rho: X \to X$ be a function defined by $\rho(x) := 1-(x \mod 2)$, \ie $\rho(x) = 1$ if $x$ is even whereas $\rho(x) = 0$ if $x$ is odd.
%One can check that $\rho$ satisfies \ref{item1:three-properties-of-rep-qmap}: for any $x,x' \in X$, $\rho(x) = \rho(x')$ if and only if $x,x'$ have the same parity, \ie $G\cdot x = G \cdot x'$.
%However, $\rho$ is \emph{not} an orbit-representative map of $G \curvearrowright X$.
%This is because $\rho$ fails \ref{item2:three-properties-of-rep-qmap}: $G\cdot 0 = 2\integers$ whereas $G\cdot \rho(0) = G\cdot 1 = 2\integers+1$; $\rho$ fails \ref{item3:three-properties-of-rep-qmap} as well: $\rho\circ\rho(0) = \rho(1) = 0$ whereas $\rho(0) = 1$.
So, being an orbit-representative map is stronger than just being an orbit-labeling map (only satisfies \ref{item1:three-properties-of-rep-qmap}), but further requires the label of each orbit to be a point in that orbit, \ie indeed a representative.
However, it turns out that the latter is the only missing piece.
As a result, we have the following converse theorems of Theorem~\ref{thm:three-properties-of-rep-qmap}.

\begin{theorem}\label{thm:rep-qmap-def-alt1}
Let $G \curvearrowright X$ be a $G$-action on $X$.
A function $\rho: X \to X$ is an orbit-representative map of $G \curvearrowright X$ if it satisfies \ref{item1:three-properties-of-rep-qmap} \ref{item2:three-properties-of-rep-qmap} in Theorem~\ref{thm:three-properties-of-rep-qmap}.
\end{theorem}

\begin{theorem}\label{thm:rep-qmap-def-alt2}
Let $G \curvearrowright X$ be a $G$-action on $X$.
A function $\rho: X \to X$ is an orbit-representative map of $G \curvearrowright X$ if it satisfies \ref{item1:three-properties-of-rep-qmap} \ref{item3:three-properties-of-rep-qmap} in Theorem~\ref{thm:three-properties-of-rep-qmap}.
\end{theorem}

\begin{remark}
Theorems~\ref{thm:rep-qmap-def-alt1} and \ref{thm:rep-qmap-def-alt2} provide two alternative and \emph{non-constructive} ways of proving a function is an orbit-representative map.
Their proofs are relegated to Appendices~\ref{app:rep-qmap-def-alt1} and \ref{app:rep-qmap-def-alt2}, respectively.
\end{remark}

To solve our restricted orbit computation problem~\eqref{eqn:orbit-computation-practical-special}, it suffices to have an algorithmic implementation of an orbit-labeling map.
Yet later (Section~\ref{sec:stage-1}), we will see that having an orbit-representative map, as a special orbit-labeling map, works nicely in the middle of our two-stage algorithm, regarding nice group actions (like translations) on nice infinite spaces (like $\ambsp$).

\section{Algorithmic Roadmap}
\label{sec:algorithmic-roadmap}

Thus far, we have described our special orbit computation problem in terms of its inputs and desired output.
We are now ready to present an specialized algorithm to solve this problem.
Referring to the problem formulation~\eqref{eqn:orbit-computation-practical-special}, our algorithm is an implementation of an orbit-labeling map $\lambda$ (output) of the group action $\langle S \rangle \curvearrowright \ambsp$ ($S$ is the first input) on the subset $Z \subseteq \ambsp$ (the second input).

The complete algorithm divides into several steps.
In this section, we first present an algorithmic roadmap to walk readers through the big picture, where we outline the major steps at a high level in their execution order.
Clearly, the nature of orbit computation is to identify the equivalence relation $\sim_{\langle S \rangle}$ induced by $\langle S \rangle \curvearrowright \ambsp$, where we say two points are related by an isometry in $\langle S \rangle$ (denoted $x \sim_{\langle S \rangle} x'$) if and only if there exists an $h \in \langle S \rangle$ such that $x' = h(x)$.
However, if we further have the semidirect-product decomposition $\langle S \rangle = \left[ \tcal_{\langle S \rangle} \circ \rcal_{\langle S \rangle} \right]$, the above isometry equivalence $\sim_{\langle S \rangle}$ can be decomposed accordingly.
So, the main idea here in the roadmap is to divide the whole procedure into two stages: translation equivalence $\sim_{\tcal_{\langle S \rangle}}$ first, and rotation equivalence $\sim_{\rcal_{\langle S \rangle}}$ in succession.

\begin{enumerate}[label=$\blacksquare~$, itemsep=-0.03in]
\item In the first stage (Step 1-3), we consider the \emph{translation equivalence} $\sim_{\tcal_{\langle S \rangle}}$ on $\ambsp$: $x \sim_{\tcal_{\langle S \rangle}} x' \iff \tcal_{\langle S \rangle} \cdot x = \tcal_{\langle S \rangle} \cdot x'$.
Restricting our attention to the subgroup $\tcal_{\langle S \rangle} \trianglelefteq \langle S \rangle$, we derive an explicit formula for an orbit-representative map $\rho_\tcal$ of the translation subgroup action $\tcal_{\langle S \rangle} \curvearrowright \ambsp$.
\item Building on the first stage, in the second stage (Step 4-5), we further consider a \emph{rotation equivalence} $\sim_{\rcal_{\langle S \rangle}}$ on ${\rm Im}\rho_\tcal$.
It is important to point out that, rather than on $\ambsp$, this rotation equivalence $\sim_{\rcal_{\langle S \rangle}}$ is a new equivalence relation on ${\rm Im}\rho_\tcal$, \ie a fundamental domain of $\tcal_{\langle S \rangle}\curvearrowright \ambsp$ representing the space of equivalence classes under $\sim_{\tcal_{\langle S \rangle}}$.
Roughly speaking, we superimpose $\sim_{\rcal_{\langle S \rangle}}$ among the equivalence classes under $\sim_{\tcal_{\langle S \rangle}}$, and merge two equivalence classes in the first stage if they are equivalent in the second stage.
Instead of an explicit formula, the while-loop implements the rotation equivalence $\sim_{\rcal_{\langle S \rangle}}$.
\end{enumerate}

\vspace{0.1in}
\noindent
\fbox{\parbox{\textwidth}{
\vspace{0.1in}\centerline{\bf Algorithmic Roadmap}
\begin{enumerate}[label=Step~\arabic*:, leftmargin=0.6in]
\item Compute the basis matrix $B$ for $\tcal_{\langle S \rangle}$, \ie the matrix\\
$B = [\bm{b_1}, \ldots, \bm{b_m}] \in \integers^{n\times m}$
with $\{t_{\bm{b_1}}, \ldots, t_{\bm{b_m}}\}$ being a basis of $\tcal_{\langle S \rangle}$.
\item Compute the pseudoinverse $B^\dagger := (B^\top B)^{-1}B^\top$.
\item Compute the orbit-representative map $\rho_{\tcal}(x) := x-B\lfloor B^\dagger x\rfloor$, for \\every $x \in Z$, after which the set $\rho_{\tcal}(Z)$ is computed.
\item Execute the while-loop below to obtain a label for every $x \in \rho_{\tcal}(Z)$:

\begin{minipage}{.7\textwidth}
\hspace{0.4in}
\begin{algorithm}[H]
  \While{$\rho_{\tcal}(Z)$ is not empty}{
    pick $\omega \in \rho_{\tcal}(Z)$\;
    compute the set: \\
    \hspace{0.1in} $C_\omega := \{r_{R\star}(\omega) \mid r_R \in \rcal_{\langle S \rangle}\} \cap \rho_{\tcal}(Z)$\\
    \hspace{0.1in} where $r_{R\star} := \rho_{\tcal} \circ r_R|_{{\rm Im}\rho_{\tcal}}$\;
    label every element in $C_\omega$ by $\omega$\;
    remove elements in $C_\omega$ from $\rho_{\tcal}(Z)$, \ie \\
    \hspace{0.1in} $\rho_{\tcal}(Z) \leftarrow \rho_{\tcal}(Z) \backslash C_\omega$\;
  }
\end{algorithm}
\end{minipage}

\item Label every $x \in Z$ by the label of $\rho_{\tcal}(x)$.
\end{enumerate}
}}

\vspace{0.1in}
The entire procedure in the algorithmic roadmap, including both stages, can be executed solely on the input subset $Z \subseteq \ambsp$.
In the end, we will show that composing the two functions implemented by the two stages yields an orbit-labeling map of $\langle S \rangle \curvearrowright \ambsp$ on $Z$, which finishes solving the orbit computation problem~\eqref{eqn:orbit-computation-practical-special}.

Next, Sections~\ref{sec:stage-1} and \ref{sec:stage-2} will respectively delve into the two stages: first considering \emph{translation equivalence}, then adding \emph{rotation equivalence}.
These two sections together will cover the details in each individual step, while in the same pass, prove that the algorithmic output equals the desired output defined in the orbit computation problem~\eqref{eqn:orbit-computation-practical-special}, verifying the correctness of the algorithm.

\section{Algorithmic Stage 1: Translation Equivalence (Step 1-3)}
\label{sec:stage-1}

Given any finite and atomic subset $S \subseteq \iz$ as input, consider $\tcal_{\langle S \rangle} = \langle S \rangle \cap \tz$ consisting of all the translations in $\langle S \rangle$.
In Stage 1, we derive an explicit formula for an orbit-representative map of $\tcal_{\langle S \rangle} \curvearrowright \ambsp$ capturing the translation equivalence $\sim_{\tcal_{\langle S \rangle}}$ on $\ambsp$.
This is started by finding a basis of $\tcal_{\langle S \rangle}$.

\subsection{Finding a Basis of $\tcal_{\langle S \rangle}$ (Step 1)}
\label{sec:step-1}

\begin{lemma}\label{lemma:translations-in-s-gen-is-free-z-module}
$\tcal_{\langle S \rangle}$ is a free $\integers$--module (of rank $\leq n$).
\end{lemma}
\begin{proof}
Both $\tcal_{\langle S \rangle}$ and $\tz$ are abelian, so both are $\integers$--modules;
in particular, $\tcal_{\langle S \rangle}$ is a submodule of $\tz$ since  $\tcal_{\langle S \rangle} \subseteq \tz$.
Note that $\tz \cong \ambsp$ is a free $\integers$--module (of rank $n$) and $\integers$ is a PID (principal ideal domain).
These two facts imply that $\tcal_{\langle S \rangle}$ is free (of rank $\leq n$).
\end{proof}

\begin{remark}
Lemma~\ref{lemma:translations-in-s-gen-is-free-z-module} ensures that it makes sense to talk about a basis of $\tcal_{\langle S \rangle}$ as a free $\integers$--module (as well as its rank).
\end{remark}

To find a basis of  $\tcal_{\langle S \rangle}$, we start from a generating set of it, namely $\tcal_S^+$ as defined earlier in Equation~\eqref{eqn:atomically-generated-subgroup-tran-zn-char}.
By definition,
\begin{align*}
\tcal_S^+ := (\tcal_S)^{\rcal_{\langle S \rangle}} &= \{r_R \circ t_v \circ r_R^{-1} \mid t_v \in \tcal_S, r_R \in \rcal_{\langle S \rangle} \} \\
&= \{t_{Rv} \mid t_v \in \tcal_S, r_A \in \rcal_{\langle S \rangle} \}.
\end{align*}
Note that $\tcal_S = S \cap \tz$ is finite since $S$ is finite; further, $\rcal_{\langle S \rangle} = \langle S \rangle \cap \rz$ is also finite since $\rz$ is finite (recall: $|\rz| = |\nz|\cdot |\pz| = 2^n(n!)$).
Therefore, $\tcal_S^+$ is finite, and elements in $\tcal_S^+$ can be algorithmically enumerated by a nested loop with the outer loop enumerating $t_v \in \tcal_S$ and the inner loop enumerating $r_R \in \rcal_{\langle S \rangle}$ and then taking the conjugate.

Enumerating $t_v \in \tcal_S$ is clear, since $\tcal_S$ is given (translations in $S$).
To enumerate $r_R \in \rcal_{\langle S \rangle}$, we leverage its semidirect-product decomposition expressed in Equation~\eqref{eqn:atomically-generated-subgroup-rota-zn-binary-semidirect-product-decomp} \ie $\rcal_{\langle S \rangle} = \left[ \ncal_{\langle S \rangle} \circ \pcal_{\langle S \rangle} \right]$.
So, the enumeration of $r_R \in \rcal_{\langle S \rangle}$ can also be done by a nested loop with the outer loop enumerating $r_N \in \ncal_{\langle S \rangle}$ and the inner loop enumerating $r_P \in \pcal_{\langle S \rangle}$ and then taking the composition.

Enumerating $r_P \in \pcal_{\langle S \rangle}$ is done by $\pcal_{\langle S \rangle} = \langle \pcal_{S} \rangle$ in Equation~\eqref{eqn:atomically-generated-subgroup-perm-zn-char}, \ie computing the permutation subgroup generated by $\pcal_S$.
This is a well-studied computational problem \eg via the computation of a base and strong generating set (BSGS)~\cite{HoltEO2005, Sims1970, Knuth1991}, and we directly plug in an off-the-shelf implementation (\eg from GAP or MAGMA).
Notably, in special cases when the dimensionality $n$ is small, we can always precompute---a one-time computation---and cache the full family $\mathbb{P}$ that comprises all subgroups of $\pz$ (this is basically the same as enumerating all subgroups of the symmetric group ${\rm Sym}_n$ since $\pz \cong {\rm Sym}_n$), then $\langle \pcal_S \rangle$ by definition, is the smallest member in $\mathbb{P}$ containing $\pcal_S$.

Enumerating $r_N \in \ncal_{\langle S \rangle}$ is done by $\ncal_{\langle S \rangle} = \langle \ncal_S^+ \rangle$ in Equation~\eqref{eqn:atomically-generated-subgroup-nega-zn-char}, where
\begin{align*}
\ncal_S^+ := (\ncal_S)^{\pcal_{\langle S \rangle}} &= \{r_P \circ r_N \circ r_P^{-1} \mid r_N \in \ncal_S, r_P \in \pcal_{\langle S \rangle} \} \\
&= \{r_{PNP^{-1}} \mid r_N \in \ncal_S, r_P \in \pcal_{\langle S \rangle} \}.
\end{align*}
Like the enumeration of $\tcal_S^+$, elements in $\ncal_S^+$ can be algorithmically enumerated by a nested loop with the outer loop enumerating $r_N \in \ncal_S$ (given) and the inner loop enumerating $r_P \in \pcal_{\langle S \rangle}$ (described earlier) and then taking the conjugate.

Now the final question: once we have enumerated elements in $\ncal_S^+$, how do we enumerate elements in $\ncal_{\langle S \rangle} = \langle \ncal_S^+ \rangle$?
The catch here is the following lemma.
\begin{lemma}
\label{lemma:negation-subgroup-free-z2-module}
$\nz \cong \integers_2^n$ is an $n$-dimensional vector space over $\integers_2$, and $\ncal_{\langle S \rangle}$ is a subspace (of $\nz$) of dimension $\leq n$. Then clearly, $|\ncal_{\langle S \rangle}| = 2^{\rm{dim}(\ncal_{\langle S \rangle})}$.
\end{lemma}
\noindent
Therefore, from the generating set $\ncal_S^+$ and via a variant of Gaussian Elimination (over $\integers_2$), we can compute a basis of $\ncal_{\langle S \rangle} = \langle \ncal_S^+ \rangle$, and every element in $\ncal_{\langle S \rangle}$ can be identified by a unique (boolean) linear combination of the basis.

So far, in a backward manner, we have introduced the whole pathway of computing a generating set of $\tcal_{\langle S \rangle}$, namely $\tcal_S^+$, which consists of several nested subroutines.
In the last step, we compute a basis of $\tcal_{\langle S \rangle}$ from $\tcal_S^+$.
The complete computational flow is then summarized in Figure~\ref{fig:flow-of-computing-basis-of-translations-in-s}.
\begin{figure}[t]
\begin{center}
\includegraphics[width=0.61\columnwidth]{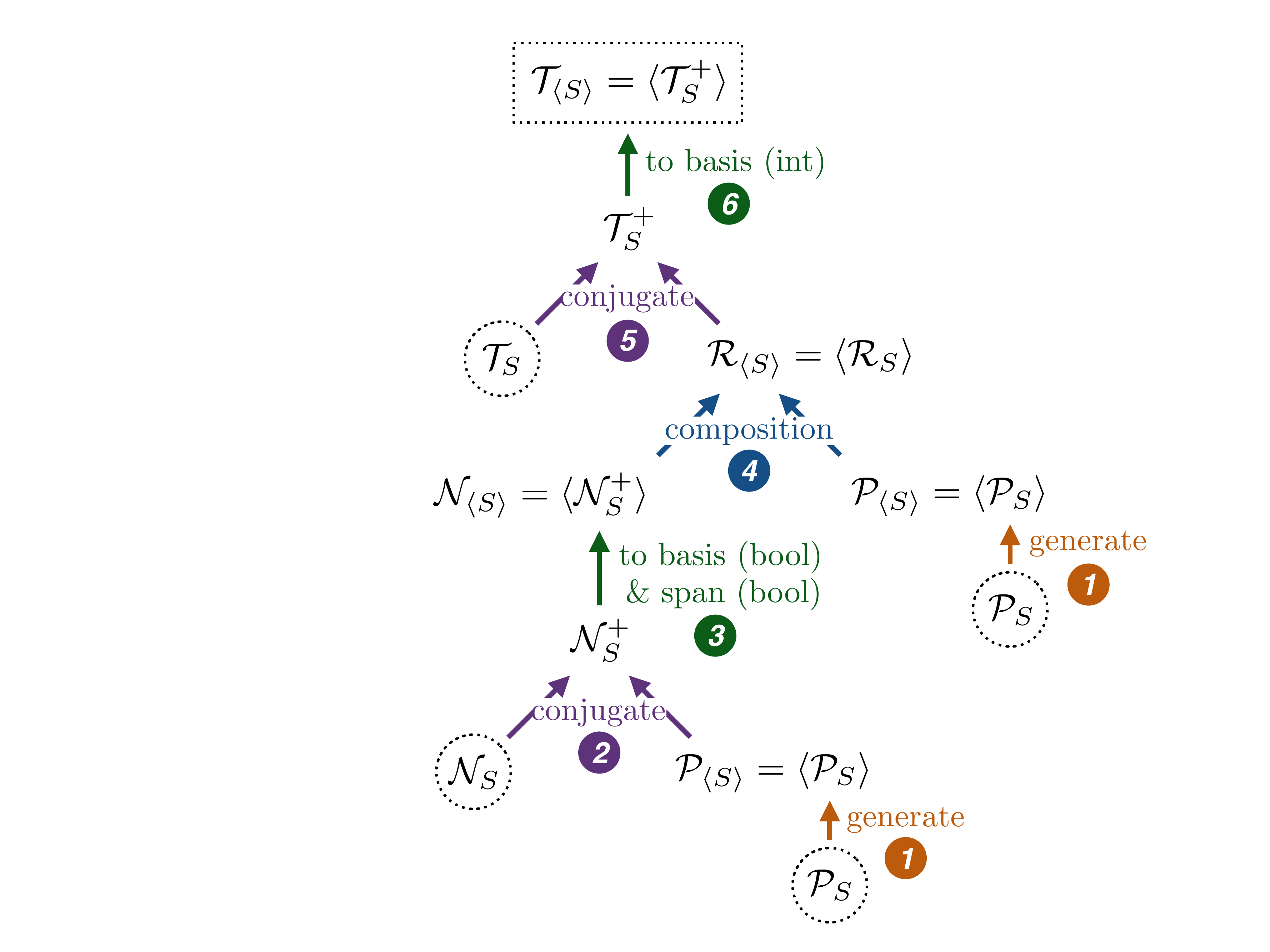}
\end{center}
\caption{The complete computational flow of computing a basis of $\tcal_{\langle S \rangle}$.
The dotted circles mark the inputs; the dotted rectangle marks the output.
Numbers mark the execution order.}
\label{fig:flow-of-computing-basis-of-translations-in-s}
\end{figure}
To finish the last piece, we detail the computation of a basis from a generating set of a free module.
There are two such cases involved in the complete computational flow (colored green in Figure~\ref{fig:flow-of-computing-basis-of-translations-in-s}): one is in the middle, for the $\integers_2$--vector space $\ncal_{\langle S \rangle} = \langle \ncal_S^+ \rangle$; the other is in the end, for the free $\integers$--module $\tcal_{\langle S \rangle} = \langle \tcal_S^+ \rangle$.

\begin{enumerate}[label=$\blacksquare~$, itemsep=-0.03in]
\item $\ncal_S^+ \longrightarrow$ a basis of $\langle \ncal_S^+ \rangle$: \emph{boolean row reduction}.

\begin{enumerate}[label=$\rhd~$, itemsep=-0.03in]
\item \textbf{Standard approach: use the group.}
We compute $\ncal_S^+ := (\ncal_S)^{\langle \pcal_S \rangle}$ explicitly from all pairs in $\ncal_S \times \langle \pcal_S \rangle$, so the full group $\langle \pcal_S \rangle$ is used.
After computing $\ncal_S^+$, we run the boolean version of Gaussian Elimination ($\mbox{GE}_b$) to get the unique \emph{reduced row echelon form} of the matrix obtained from aligning the boolean vectors that isomorphically correspond to the generators in $\ncal_S^+$ row by row.
We write $\mbox{GE}_b(\ncal_S^+)$ to denote the result of this process.
\underline{Note:} in the boolean version, Gaussian Elimination is extremely easy, since there are only two types of elementary row operations:
a) swapping two rows;
b) adding (\ie modulo-$2$ addition, or logical XOR) one row to another.
\item \textbf{Improved approach: use only generators.}
We never explicitly list all elements in $\ncal_S^+$, and we only use $\pcal_S$ instead of $\langle \pcal_S \rangle$ in the process of computing a basis of $\langle \ncal_S^+ \rangle$.
As a trade off, instead of one big $\mbox{GE}_b$ run, the whole process may contain multiple but smaller $\mbox{GE}_b$ runs in succession.
More specifically, we initialize $\bcal$ to be $\mbox{GE}_b(\ncal_S)$ and let $\pcal_S^* := \pcal_S \cup \{e\}$.
Next, we execute a while-loop: as long as $\bcal \neq \mbox{GE}_b(\bcal^{\pcal_S^*})$, we update $\bcal$ into $\mbox{GE}_b(\bcal^{\pcal_S^*})$; otherwise, we stop.
It is easy to check that whenever $\bcal = \mbox{GE}_b(\bcal^{\pcal_S^*})$, $\bcal$ is the desired basis, \ie $\bcal = \mbox{GE}_b(\ncal_S^+)$.
Further, it is easy to check that whenever $\bcal \neq \mbox{GE}_b(\bcal^{\pcal_S^*})$, $| \mbox{GE}_b(\bcal^{\pcal_S^*}) | - |\bcal| \geq 1$.
Since $\mbox{dim}(\langle \ncal_S^+ \rangle) \leq n$ (Lemma~\ref{lemma:negation-subgroup-free-z2-module}), the update of $\bcal$ occurs at most $n-1$ times.
Hence, there are at most $n$ $\mbox{GE}_b$ runs, each of which is on a much smaller generating set compared to $\ncal_S^+$; further, not only the explicit enumeration of $\ncal_S^+$ but also that of $\langle \pcal_S \rangle$ are avoided.
\end{enumerate}

\item $\tcal_S^+ \longrightarrow$ a basis of $\langle \tcal_S^+ \rangle$: \emph{integral row reduction}.

\begin{enumerate}[label=$\rhd~$, itemsep=-0.03in]
\item \textbf{Standard approach: use the group.}
We compute $\tcal_S^+ := (\tcal_S)^{\langle \rcal_S \rangle}$ explicitly from all pairs in $\tcal_S \times \langle \rcal_S \rangle$, so the full group $\langle \rcal_S \rangle$ is used.
After computing $\tcal_S^+$, we run the Lenstra-Lenstra-Lov\'{a}sz (LLL) Algorithm~\cite{LenstraLL1982} to get the Hermite normal form (reduced echelon form for matrices over $\integers$) of the matrix obtained from aligning the translation vectors of the generators in $\tcal_S^+$ row by row.
We write $\mbox{LLL}(\tcal_S^+)$ to denote the result of this process.
\underline{Note:} integral row reduction is computationally more expensive than boolean row reduction.
\item \textbf{Possibly improved: use only generators.}
We never explicitly list all elements in $\tcal_S^+$, and we only use $\rcal_S$ instead of $\langle \rcal_S \rangle$ in the process of computing a basis of $\langle \tcal_S^+ \rangle$.
As a trade off, instead of one big LLL run, the whole process may contain multiple but smaller LLL runs in succession.
Similar to the improved approach in the case of negations, we initialize $\bcal$ to be $\mbox{LLL}(\tcal_S)$ and let $\rcal_S^* := \rcal_S \cup \{e\}$.
Next, we execute a while-loop: as long as $\bcal \neq \mbox{LLL}(\bcal^{\rcal_S^*})$, we update $\bcal$ into $\mbox{LLL}(\bcal^{\rcal_S^*})$; otherwise, we stop.
It is again easy to check: whenever $\bcal = \mbox{LLL}(\bcal^{\rcal_S^*})$, $\bcal$ is the desired basis, \ie $\bcal = \mbox{LLL}(\tcal_S^+)$.
However, unlike what we saw in a $\integers_2$--vector space, here we generally do not have the guarantee on rank-increase when an update of $\bcal$ occurs.
We can certainly have $|\bcal| = |\mbox{LLL}(\bcal^{\rcal_S^*})|$ even if $\bcal \neq \mbox{LLL}(\bcal^{\rcal_S^*})$.
So, it is not immediately clear what is an upper bound on the number of updates (and the number of LLL runs) which may certainly exceed $n$.
Yet, the fact that we have avoided the explicit enumeration of $\langle \rcal_S \rangle$ renders the algorithm feasible for much larger $n$;
further in practice, the fact that we have avoided the explicit enumeration of $\tcal_S^+$ may benefit from running smaller-scale LLL processes, since a single LLL run on a large $\ncal_S^+$ is computationally very expensive.
\end{enumerate}

\end{enumerate}

\subsection{Computing the Translation Equivalence $\sim_{\tcal_{\langle S \rangle}}$ on $\ambsp$ (Step 2-3)}
\label{sec:step-2-3}

Let $\{t_{\bm{b_1}}, \ldots, t_{\bm{b_m}}\}$ be the basis of $\tcal_{\langle S \rangle}$ computed in the previous subsection, and call $B := [\bm{b_1}, \ldots, \bm{b_m}] \in \integers^{n\times m}$ the basis matrix of $\tcal_{\langle S \rangle}$.
Notably, $B^\top$ is the (row-style) Hermite normal form produced by the LLL algorithm.
The column space of the basis matrix $B\integers^m := \{B\mu \mid \mu \in \integers^m\}$ is the set comprising all linear combinations (with integral coefficients) of $\{\bm{b_1}, \ldots, \bm{b_m}\}$.
So, $B\integers^m \cong \tcal_{\langle S \rangle}$, and $B\integers^m$ is the set comprising all translation vectors found in $\tcal_{\langle S \rangle}$, \ie
\begin{align}\label{eqn:translation-vectors-in-basis}
v \in B\integers^m \iff t_v \in \tcal_{\langle S \rangle}.
\end{align}

One can check: the set $\{\bm{b_1}, \ldots, \bm{b_m}\}$ is linearly independent over $\reals$.
(\underline{Note:} $\{\bm{b_1}, \ldots, \bm{b_m}\}$ as a basis of a free $\integers$-module \emph{only} assures independence over $\integers$;
additional steps are needed to show independence over $\reals$ but not hard.)
Hence, as a real matrix, the basis matrix $B$ has full column rank, \ie $\mbox{rank}(B) = m$.
This further implies that $B^\top B$ is invertible~\cite{BoydV2018}.
We denote the pseudoinverse of $B$ by $B^\dagger := (B^\top B)^{-1}B^\top$, and review three known facts about it~\cite{BoydV2018}.
First, $B^\dagger$ is a left inverse, \ie $B^\dagger B = I$.
Second, for any $x \in \reals^n$, $BB^\dagger x$ is the Euclidean projection of $x$ onto the column space of $B$ (\ie the closest point in $B\reals^m$ to $x$), where $B^\dagger x$ gives the coefficients.
Third, any $x \in \reals^n$ can be decomposed into two orthogonal parts as follows: $x = (x - BB^\dagger x) + BB^\dagger x$.
%such that $\|x\|_2^2 = \|x - BB^\dagger x\|_2^2 + \|BB^\dagger x\|_2^2$.

\begin{definition}\label{def:rho-t}
Define $\rho_\tcal: \ambsp \to \ambsp$ by the following formula
\begin{align}\label{eqn:rho-t}
\rho_\tcal(x) := x - B\lfloor B^\dagger x\rfloor \quad \mbox{ for any } x \in \ambsp,
\end{align}
where $\lfloor \cdot \rfloor$ takes the floor function coordinate-wise.
\end{definition}

\begin{remark}
Figure~\ref{fig:floored-project} depicts the operator $B\lfloor B^\dagger \cdot \rfloor: \reals^n \to B \integers^m$:
\begin{figure}[h!]
\begin{center}
\includegraphics[width=0.5\columnwidth]{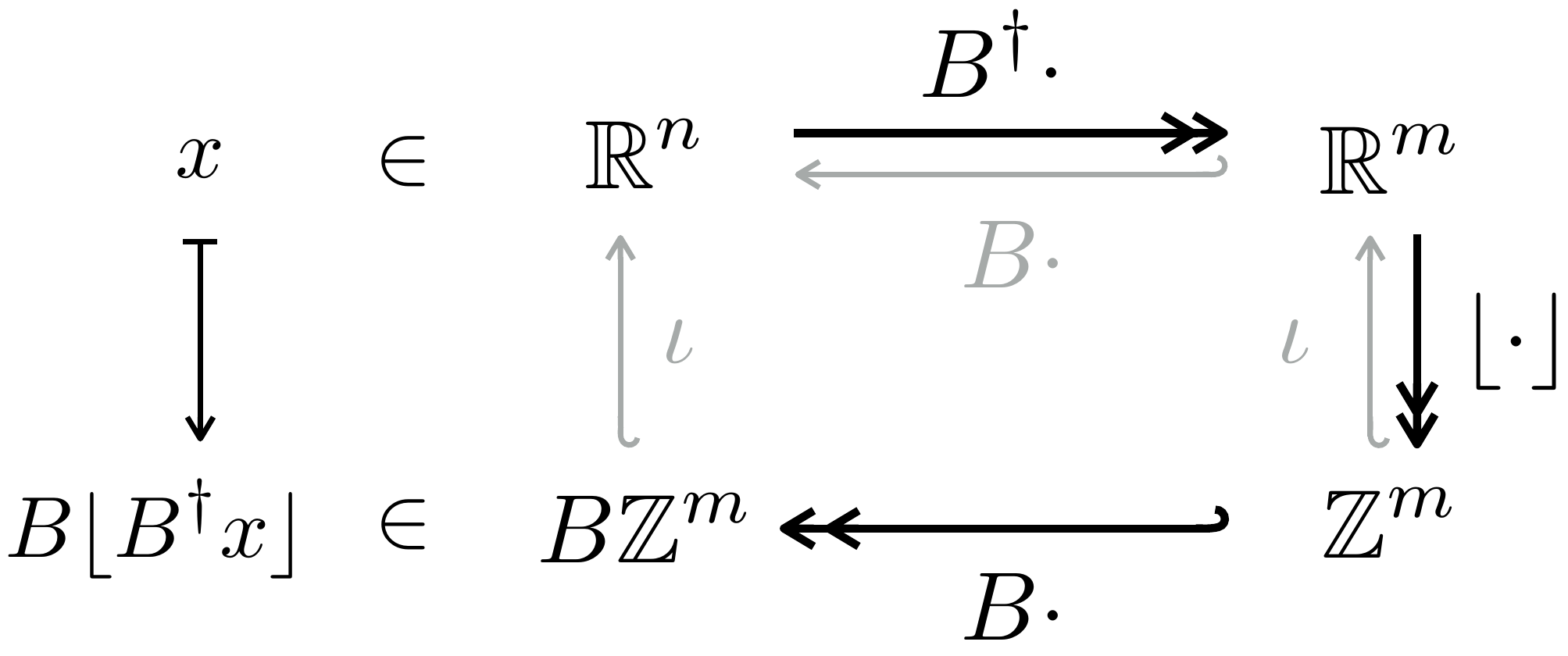}
\end{center}
\caption{A commutative diagram illustrating the operator $B\lfloor B^\dagger \cdot \rfloor: \reals^n \to B \integers^m$.}
\label{fig:floored-project}
\end{figure}
\end{remark}

\begin{theorem}\label{thm:rho-t-is-an-orbit-representative-map}
$\rho_\tcal$ is an orbit-representative map of $\tcal_{\langle S \rangle} \curvearrowright \ambsp$.
\end{theorem}
\begin{proof}
We first check the orbit-labeling property. Pick any $x, x' \in \ambsp$.
Suppose $\tcal_{\langle S \rangle} \cdot x = \tcal_{\langle S \rangle} \cdot x'$, then there exists a translation $t_v \in \tcal_{\langle S \rangle}$ such that $x' = t_v(x) = x+v$.
By Condition~\eqref{eqn:translation-vectors-in-basis}, $v \in B\integers^m$, \ie $v = B\mu$ for some $\mu \in \integers^m$.
So,
\begin{align*}
\rho_\tcal(x') &= x' - B\lfloor B^\dagger x'\rfloor \\
&= x+B\mu - B\lfloor B^\dagger (x+B\mu)\rfloor \\
&= x+B\mu - B\lfloor B^\dagger x + \mu\rfloor \\
&= x+B\mu - B(\lfloor B^\dagger x\rfloor + \mu) \\
&= x - B\lfloor B^\dagger x\rfloor = \rho_\tcal(x).
\end{align*}
Conversely, suppose $\rho_\tcal(x) = \rho_\tcal(x')$, then $x - B\lfloor B^\dagger x\rfloor = x' - B\lfloor B^\dagger x'\rfloor$. Rearranging the terms yields $x' = x + B\mu = t_{B\mu}(x)$ with $\mu = \lfloor B^\dagger x'\rfloor-\lfloor B^\dagger x\rfloor \in \integers^m$.
By Condition~\eqref{eqn:translation-vectors-in-basis}, $t_{B\mu} \in \tcal_{\langle S \rangle}$.
Therefore, $\tcal_{\langle S \rangle}\cdot x = \tcal_{\langle S \rangle}\cdot x'$.

We next check the representative property.
By definition, for any $x \in X$, $\rho_\tcal(x) = x - B\lfloor B^\dagger x \rfloor = t_{- B\lfloor B^\dagger x \rfloor}(x)$ where the translation vector $- B\lfloor B^\dagger x \rfloor \in B\integers^m$.
By Condition~\eqref{eqn:translation-vectors-in-basis}, $t_{- B\lfloor B^\dagger x \rfloor} \in \tcal_{\langle S \rangle}$.
This implies $\tcal_{\langle S \rangle}\cdot \rho_\tcal(x) = \tcal_{\langle S \rangle}\cdot x$.
Thus, by Theorem~\ref{thm:rep-qmap-def-alt1}, $\rho_\tcal$ is an orbit-representative map of the translation subgroup action $\tcal_{\langle S \rangle} \curvearrowright \ambsp$.
\end{proof}

\begin{figure}[t]
\begin{center}
\includegraphics[width=0.99\columnwidth]{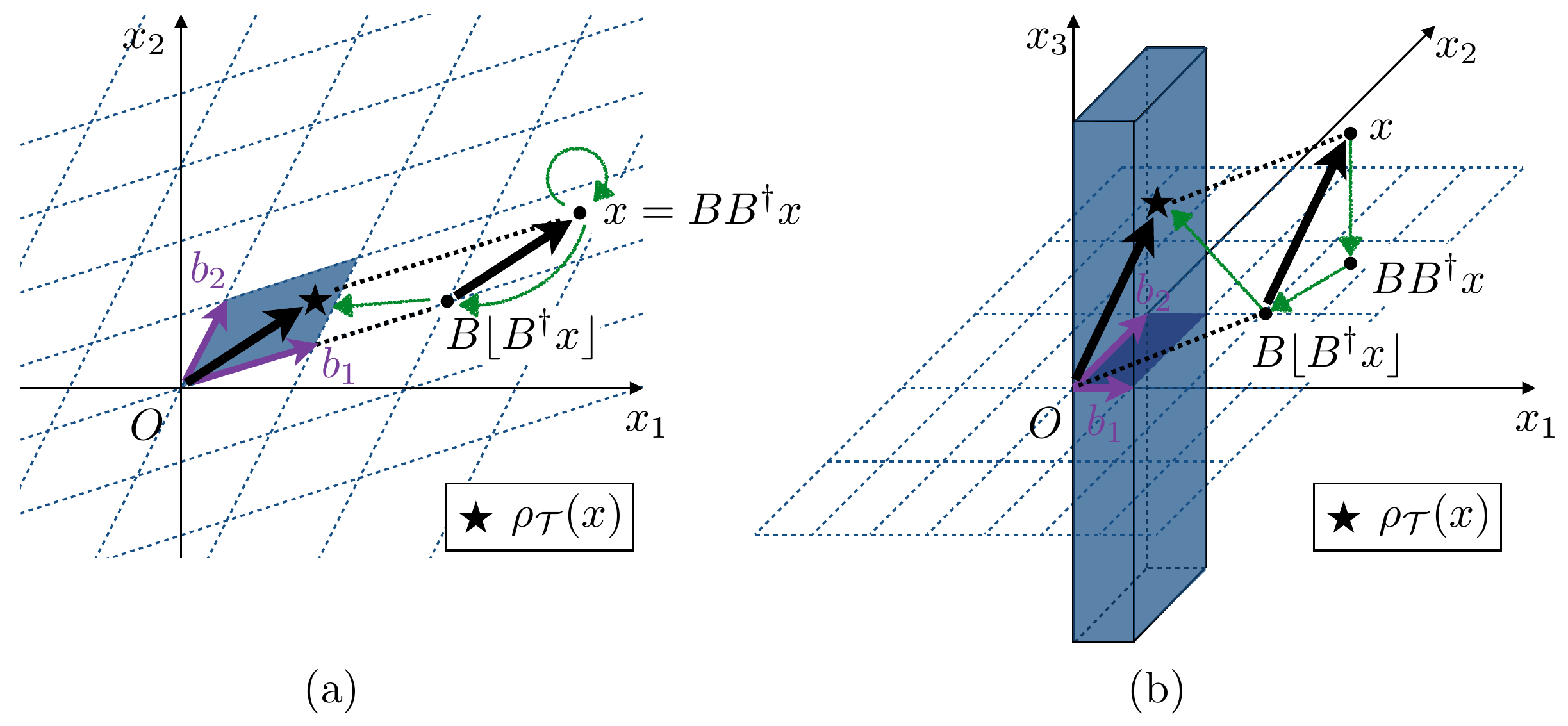}
\end{center}
\caption{Two examples for $\rho_\tcal: \ambsp \to \ambsp$ with some basis matrix $B = [\bm{b_1} ~ \bm{b_2}] \in \integers^{n\times 2}$: (a) $n=2$; (b) $n = 3$.
In both examples, $\rho_\tcal$ is an orbit-representative map of ${\rm span}(\{t_{\bm{b_1}},t_{\bm{b_2}}\}) \curvearrowright \ambsp$, and all integral points in the shaded area---the parallelogram in (a) and the infinite cube in (b)---form the fundamental domain ${\rm Im}\rho_\tcal$.
Given any $x \in \ambsp$, $\rho_\tcal(x)$ is a projection (in the sense of an idempotent) of $x$ onto the fundamental domain. This projection is pictorially dissected into the process (green path): $x \mapsto BB^\dagger x \mapsto B\lfloor B^\dagger x\rfloor \mapsto x - B\lfloor B^\dagger x\rfloor =: \rho_\tcal(x)$.
}
\label{fig:rho-t-2d-3d}
\end{figure}

We give two examples to illustrate the function $\rho_\tcal$ defined in Definition~\ref{def:rho-t}.
In the first example (Figure~\ref{fig:rho-t-2d-3d}a), we consider $\integers^2$ and $B \in \integers^{2 \times 2}$, where the pseudoinverse of $B$ is the inverse, \ie $B^\dagger = B^{-1}$; in the second example (Figure~\ref{fig:rho-t-2d-3d}b), we consider $\integers^3$ and $B \in \integers^{3 \times 2}$, where the pseudoinverse is only a left inverse.
In both examples, we dissect the mapping $x \mapsto \rho_\tcal(x)$ in the following process: $x \mapsto BB^\dagger x \mapsto B\lfloor B^\dagger x\rfloor \mapsto x - B\lfloor B^\dagger x\rfloor =: \rho_\tcal(x)$.

\begin{remark}
The following are immediate from the general properties of an orbit-representative map mentioned in Remark~\ref{remark:three-properties-of-rep-qmap}.
\begin{enumerate}[label=$\blacksquare~$, itemsep=-0.03in]
\item ${\rm Im}\rho_\tcal$ is a fundamental domain of $\tcal_{\langle S \rangle} \curvearrowright \ambsp$.
\item $\rho_\tcal$ is a projection onto ${\rm Im}\rho_\tcal$ in the sense of an idempotent: $\rho_\tcal \circ \rho_\tcal = \rho_\tcal$.
\item For any $x \in \ambsp$, $\rho_\tcal(x)$ is a representative of the orbit $\tcal_{\langle S \rangle} \cdot x$.
\end{enumerate}
\end{remark}

\section{Algorithmic Stage 2: Rotation Equivalence (Step 4-5)}
\label{sec:stage-2}

%So far, we have identified that two points are equivalent if they are related by a translation in $\tcal_{\langle S \rangle}$.
%We further identify that two equivalence classes of points are themselves equivalent if they are related by a rotation in $\rcal_{\langle S \rangle}$.
Since every isometry in $\langle S \rangle$ is uniquely represented by a translation in $\tcal_{\langle S \rangle}$ and a rotation in $\rcal_{\langle S \rangle}$ (\ie the atomic property and Equation~\eqref{eqn:atomically-generated-subgroup-isom-zn-binary-semidirect-product-decomp}), intuitively, merging classes of translation equivalent points that are further rotation equivalent yields the set of equivalence classes under $\sim_{\langle S \rangle}$, \ie the set of orbits under $\langle S \rangle \curvearrowright \ambsp$ (which is what we want).
In this section, we will make this intuition precise, while at the same time, elaborate Stage 2 in the algorithmic roadmap.
{\underline{Note:}} most conclusions in this section are pretty standard, whose proofs are straightforward thus omitted.
First in a general setting, we introduce \emph{induced quotient action}, which will be used later to depict rotation equivalence on the translation-equivalence classes, or on the representatives of these classes.

\subsection{Preliminary: Induced Quotient Action}
\label{sec:prelim-induced-quotient-action}

Let $G\curvearrowright X$ be a group action denoted by $\cdot$ and $A \trianglelefteq G$ be a normal subgroup.

\begin{theorem}\label{thm:induced-quotient-action}
Define $\bullet: G/A \times X/A \to X/A$ by
\begin{align}\label{eqn:induced-quotient-action}
(gA) \bullet (A\cdot x) := A\cdot (g\cdot x),
\end{align}
then $\bullet$ is a group action.
We call it the \emph{induced quotient action} $G/A\curvearrowright X/A$ (induced from the group action $G \curvearrowright X$).
\end{theorem}
Let $\sim_G$ denote the equivalence relation (on $X$) associated with the group action $G \curvearrowright X$; let $\sim_{G/A}$ denote the equivalence relation (on $X/A$) associated with the induced quotient action $G/A \curvearrowright X/A$.
Next theorem connects the two.

\begin{theorem}\label{thm:induced-quotient-action-and-its-inducing-group-action}
For any $x,x' \in X$, $x \sim_G x'$ if and only if $A\cdot x \sim_{G/A} A\cdot x'$.
\end{theorem}
%\begin{proof}
%Pick any $x,x' \in X$, then
%\begin{align*}
%x \sim_G x' &\iff \exists g \in G \quad s.t. \quad x' = g\cdot x \\
%&\iff \exists g \in G, a \in A \quad s.t. \quad x' = (ag)\cdot x \\
%&\iff \exists g \in G, a \in A \quad s.t. \quad x' = a\cdot (g\cdot x) \\
%&\iff \exists g \in G \quad s.t. \quad A\cdot x' = A\cdot (g\cdot x) \\
%&\iff \exists g \in G \quad s.t. \quad A\cdot x' = (gA) \bullet (A\cdot x) \\
%&\iff \exists gA \in G/A \quad s.t. \quad A\cdot x' = (gA) \bullet (A\cdot x) \\
%&\iff A\cdot x \sim_{G/A} A\cdot x'.
%\end{align*}
%One can check that the second $\iff$ is true:
%\begin{enumerate}[label=, topsep=0.03in, itemsep=-0.03in, itemindent=-0.3in]
%\item $\implies$ holds since $g = a(a^{-1}g)$ for any $a \in A$;
%\item $\impliedby$ holds since $ag \in G$.
%\end{enumerate}
%\vspace{-0.13in}
%\end{proof}

We introduce two new assumptions, separately.
Under each assumption, we can rewrite the equivalence relation $\sim_{G/A}$ on $X/A$ in a new form.

\vspace{-0.05in}
\paragraph{\underline{First new assumption}}
Under the general setting ($G\curvearrowright X$, $A\trianglelefteq G$), in addition, assume that $G$ has a semidirect-product decomposition: $G = \left[ AB \right]$.
%so, we can define $\varphi_A: G \to A$ and $\varphi_B: G \to B$ such that for any $g \in G$, $g = \varphi_A(g)\varphi_B(g)$.
This allows us to rewrite the definition of $\sim_{G/A}$ as follows: for any $x,x' \in X$
\begin{align*}
A\cdot x \sim_{G/A} A\cdot x' &\iff \exists g \in G \quad s.t. \quad A\cdot x' = (gA) \bullet (A\cdot x) = A\cdot (g\cdot x) \\
&\iff \exists b \in B \quad s.t. \quad A\cdot x' = (bA) \bullet (A\cdot x) = A\cdot (b\cdot x).
\end{align*}
%One can check that the second $\iff$ is true:
%\begin{enumerate}[label=, topsep=0.03in, itemsep=-0.03in, itemindent=-0.3in]
%\item $\implies$ holds since $gA = \varphi_A(g)\varphi_B(g)A = \varphi_B(g) \varphi_B(g)^{-1} \varphi_A(g) \varphi_B(g) A = \varphi_B(g)A$ where $\varphi_B(g) \in B$;
%\item $\impliedby$ holds since $b \in B$ and $B \leq G$ implies that $b \in G$.
%\end{enumerate}

\vspace{-0.15in}
\paragraph{\underline{Second new assumption}}
Under the general setting ($G\curvearrowright X$, $A\trianglelefteq G$), in addition, assume that $\rho: X \to X$ is an orbit-representative map of $A\curvearrowright X$.
%,\ie $\rho(x) = \overline{\rho}(A\cdot x)$ for some section $\overline{\rho}$ of the orbit-quotient map of $A \curvearrowright X$ (cf.\ Definition~\ref{def:orbit-representative-map}).
This allows us to rewrite the definition of $\sim_{G/A}$ as follows: for any $x,x' \in X$
\begin{align*}
A\cdot x \sim_{G/A} A\cdot x' &\iff A\cdot \rho(x) \sim_{G/A} A\cdot x' \\
&\iff \exists g \in G  \quad s.t. \quad A\cdot x' = (gA) \bullet (A\cdot \rho(x)) = A\cdot (g\cdot \rho(x)) \\
%&\iff \exists g \in G \quad s.t. \quad \overline{\rho}(A\cdot x') = \overline{\rho}(A\cdot (g\cdot \rho(x))) \\
&\iff \exists g \in G \quad s.t. \quad \rho(x') = \rho(g\cdot \rho(x)) ~=:~ g_\star(\rho(x)).
\end{align*}
%From the first to the last $\iff$, we respectively used Theorem~\ref{thm:three-properties-of-rep-qmap}\ref{item2:three-properties-of-rep-qmap}, the definition of $\sim_{G/A}$, the fact that $\overline{\rho}$ is injective, and the definition of $\rho$.
In the last expression, we introduced a new function $g_\star$ defined as follows.

\begin{definition}
Under the general setting and the second new assumption, for any $g\in G$, we define the \emph{projected-$g$} function $g_\star: {\rm Im}\rho \to {\rm Im}\rho$ by $g_\star(\omega) = \rho(g\cdot \omega)$.
\end{definition}

\begin{remark}
As an orbit-representative map, $\rho$ is a projection onto the fundamental domain ${\rm Im}\rho$ (in the sense of an idempotent).
This explains why we call $g_\star$ the ``projected-$g$'' function, since it first applies the group action $g \cdot $ and then the projection $\rho$.
More concisely, $g_\star$ can be expressed by the diagram:
\begin{align*}
{\rm Im}\rho \xhookrightarrow{\quad \subseteq \quad} X \xrightarrow{\quad g\cdot \quad} X \xrightarrow{\quad \rho \quad} {\rm Im}\rho.
\end{align*}
\end{remark}

If both assumptions are added to the general setting, the following is immediate as a corollary of Theorem~\ref{thm:induced-quotient-action-and-its-inducing-group-action}.
We present it in its entirety as follows.
\begin{corollary}\label{coro:induced-quotient-action-and-its-inducing-group-action}
Let $G$ be a group that has a semidirect-product decomposition: $G = \left[ AB \right]$.
Let $X$ be a set, $\cdot : G \times X \to X$ be a group action $G\curvearrowright X$, and $\bullet: G/A \times X/A \to X/A$ be the induced quotient action $G/A \curvearrowright X/A$.
Further, let $\rho: X \to X$ be an orbit-representative map of $A\curvearrowright X$, and $g_\star: {\rm Im}\rho \to {\rm Im}\rho$ be the corresponding projected-$g$ function for any $g \in G$. Then, for any $x,x' \in X$,
\begin{align}
\label{eqn:induced-quotient-action-and-its-inducing-group-action-q1}
&G\cdot x = G\cdot x' \\
\label{eqn:induced-quotient-action-and-its-inducing-group-action-q2}
\iff &(G/A)\bullet (A\cdot x) = (G/A) \bullet (A\cdot x')\\
\label{eqn:induced-quotient-action-and-its-inducing-group-action-q3}
\iff &\exists b \in B \quad s.t. \quad \rho(x') = b_\star(\rho(x)).
\end{align}
\end{corollary}

\subsection{Computing the Rotation Equivalence $\sim_{\rcal_{\langle S \rangle}}$ on ${\rm Im}\rho_\tcal$ (Step 4-5)}
\label{sec:step-4-5}

We return to our main mathematical object in this paper---the subgroup $\langle S \rangle$ generated from a finite and atomic subset $S \subseteq \iz$.
Backed by Corollary~\ref{coro:induced-quotient-action-and-its-inducing-group-action}, we are now ready to introduce the rotation equivalence $\sim_{\rcal_{\langle S \rangle}}$ on the fundamental domain ${\rm Im}\rho_\tcal$ (of $\tcal_{\langle S \rangle} \curvearrowright \ambsp$), and to show that it is the last piece needed to solve our orbit computation problem~\eqref{eqn:orbit-computation-practical-special}.

Referring to the generic notations in Section~\ref{sec:prelim-induced-quotient-action}, here we specifically let $G = \langle S \rangle$, $A = \tcal_{\langle S \rangle}$, $B = \rcal_{\langle S \rangle}$, $X = \ambsp$, $\rho = \rho_\tcal$.
One can check: this case satisfies all the conditions in Corollary~\ref{coro:induced-quotient-action-and-its-inducing-group-action}.

\begin{definition}\label{def:rotation-equivalence-on-fd}
We define the \emph{rotation equivalence $\sim_{\rcal_{\langle S \rangle}}$ on ${\rm Im}\rho_\tcal$} as follows: for any $\rho_\tcal(x), \rho_\tcal(x') \in {\rm Im}\rho_\tcal$,
\begin{align}\label{eqn:rotation-equivalence-on-fd}
\rho_\tcal(x) \sim_{\rcal_{\langle S \rangle}} \rho_\tcal(x') \iff \exists r_R \in \rcal_{\langle S \rangle} \quad s.t. \quad \rho_\tcal(x') = r_{R\star}(\rho_\tcal(x)).
\end{align}
(\underline{Note:} $r_{R\star}$ is a shorthand notation for $(r_R)_{\star}$ for the sake of notational brevity.)
\end{definition}

\begin{remark}
While the translation equivalence $\sim_{\tcal_{\langle S \rangle}}$ is on $\ambsp$, the rotation equivalence $\sim_{\rcal_{\langle S \rangle}}$ is on ${\rm Im}\rho_\tcal$.
Further, $\sim_{\rcal_{\langle S \rangle}}$ is essentially the equivalence relation $\sim_{\langle S \rangle/\tcal_{\langle S \rangle}}$ associated with the induced quotient action $\langle S \rangle/\tcal_{\langle S \rangle} \curvearrowright \ambsp/\tcal_{\langle S \rangle}$.
This is immediate from \eqref{eqn:induced-quotient-action-and-its-inducing-group-action-q2}$\iff$\eqref{eqn:induced-quotient-action-and-its-inducing-group-action-q3} in Corollary~\ref{coro:induced-quotient-action-and-its-inducing-group-action} and Definition~\ref{def:rotation-equivalence-on-fd}, which collectively imply that for any $x,x'\in \ambsp$,
\begin{align*}
\rho_\tcal(x) \sim_{\rcal_{\langle S \rangle}} \rho_\tcal(x') \iff (\tcal_{\langle S \rangle}\cdot x) \sim_{\langle S \rangle/\tcal_{\langle S \rangle}} (\tcal_{\langle S \rangle}\cdot x').
\end{align*}
\end{remark}

The next theorem is immediate from \eqref{eqn:induced-quotient-action-and-its-inducing-group-action-q1}$\iff$\eqref{eqn:induced-quotient-action-and-its-inducing-group-action-q3} in Corollary~\ref{coro:induced-quotient-action-and-its-inducing-group-action}, which is the cornerstone of proving the correctness of our entire algorithmic roadmap.

\begin{theorem}\label{thm:induced-quotient-action-in-s-gen}
For any $x,x' \in \ambsp$, $\langle S \rangle \cdot x = \langle S \rangle \cdot x'$ if and only if there exists a rotation $r_R \in \rcal_{\langle S \rangle}$ such that $\rho_\tcal(x') = r_{R\star}(\rho_\tcal(x))$.
\end{theorem}

Now let us see how Steps 4 and 5 complete the whole algorithmic roadmap.
After Step 3, we have computed the set $\rho_\tcal(Z)$, \ie the projection of $Z$ onto the fundamental domain ${\rm Im}\rho_\tcal$ of the translation subgroup action $\tcal_{\langle S \rangle} \curvearrowright \ambsp$.
The while-loop in Step 4 defines a function $\rhowl: \rho_\tcal(Z) \to \rho_\tcal(Z)$ that maps every $\omega \in \rho_\tcal(Z)$ to its label $\rhowl(\omega)$.
After Step 5, every $x \in Z$ is labeled $\rhowl \circ \rho_\tcal(x)$.

\begin{theorem}\label{thm:algorithmic-roadmap-is-an-orbit-labeling-map-on-z}
The function $\lambda: = \rhowl \circ \rho_\tcal|_Z$, defined by the whole algorithmic roadmap, is an orbit-labeling map of $\langle S \rangle \curvearrowright \ambsp$ on $Z$.
\end{theorem}
\begin{proof}
Pick any $x, x' \in Z$.
First, note that $\lambda(x) = \rhowl(\rho_\tcal(x)) \in \rho_\tcal(Z)$, thus, there exists a $z \in Z$ such that $\rhowl(\rho_\tcal(x)) = \rho_\tcal(z)$.
By the definition of $\rhowl$ (the while-loop), $\rho_\tcal(x) \in C_{\rho_\tcal(z)}$.
By the definition of $C_{\rho_\tcal(z)}$, there exists a $r_R \in \rcal_{\langle S \rangle}$ such that $\rho_\tcal(x) = r_{R\star}(\rho_\tcal(z))$.
By Theorem~\ref{thm:induced-quotient-action-in-s-gen}, $\langle S \rangle \cdot x = \langle S \rangle \cdot z$.
Now following the above argument for $x' \in X$, we have
\begin{align*}
\lambda(x') = \rho_\tcal(z) &\iff \rhowl(\rho_\tcal(x')) = \rho_\tcal(z) \\
&\iff \rho_\tcal(x') \in C_{\rho_\tcal(z)} \\
&\iff \rho_\tcal(x') = r_{R'\star}(\rho_\tcal(z)) ~\mbox{ for some }~ r_{R'} \in \rcal_{\langle S \rangle}  \\
&\iff \langle S \rangle \cdot x' = \langle S \rangle \cdot z.
\end{align*}
Since $\lambda(x) = \rho_\tcal(z)$ and $\langle S \rangle \cdot x = \langle S \rangle \cdot z$, then
$\lambda(x) = \lambda(x') \iff \langle S \rangle \cdot x = \langle S \rangle \cdot x'$.
By Definition~\ref{def:orbit-labeling-map}, $\lambda$ is an orbit-labeling map of $\langle S \rangle \curvearrowright \ambsp$ on $Z \subseteq \ambsp$.
\end{proof}

The algorithmic roadmap implements the function $\lambda$, and outputs the partition $\{\hspace{0.03in}\lambda^{-1}(\{\lambda(x)\}) \mid x \in Z\hspace{0.03in}\}$ of $Z$.
Clearly, $\lambda$ satisfies the orbit-labeling property
\begin{align}\label{eqn:orbit-labeling-property-for-s-gen}
\lambda(x) = \lambda(x') \iff \langle S \rangle \cdot x = \langle S \rangle \cdot x' \quad \mbox{ for any } x, x' \in Z.
\end{align}
One can check: for any $x, x' \in Z$, the left side of \eqref{eqn:orbit-labeling-property-for-s-gen} $\iff x' \in \lambda^{-1}(\{\lambda(x)\})$; the right side of \eqref{eqn:orbit-labeling-property-for-s-gen} $\iff x' \in (\langle S \rangle \cdot x) \cap Z$.
This implies that
\begin{align*}
\{\hspace{0.03in}\lambda^{-1}(\{\lambda(x)\}) \mid x \in Z\hspace{0.03in}\} = \{\hspace{0.03in}(\langle S \rangle \cdot x) \cap Z \mid x \in Z\hspace{0.03in}\}.
\end{align*}
which verifies our algorithmic output equals the desired output in Problem~\eqref{eqn:orbit-computation-practical-special}.

\subsection{A Generators-Based Variant to Step~4 in the Algorithmic Roadmap}
\label{sec:variant-to-step-4}

Recall Step~4 is a while-loop containing a full enumeration of the rotation subgroup $\rcal_{\langle S\rangle} = \langle \rcal_S \rangle$ when computing $C_{\omega} = \{r_{R\star}(\omega) \mid r_R \in \rcal_{\langle S \rangle}\} \cap \rho_{\tcal}(Z)$.
Applying the same idea in computing a basis of $\langle \ncal_S^+ \rangle$ and $\langle \tcal_S^+ \rangle$ (\cf the end of Section~\ref{sec:step-1}), we propose a variant where we enumerate the generating set $\rcal_S$ instead of the full group $\langle \rcal_S \rangle$.
This gives an alternative way of computing $C_\omega$.
Initialize $C_\omega$ to be $\{\omega\}$.
Run an (inner) while-loop: as long as $C_\omega \neq \{r_{R\star}(\omega') \mid \omega' \in C_\omega, r_R \in \rcal_S\}$, update $C_\omega$ into $\{r_{R\star}(\omega') \mid \omega' \in C_\omega, r_R \in \rcal_S\}$; otherwise, update $C_\omega$ into $C_\omega \cap \rho_\tcal(Z)$ and stop.
One can check: $C_\omega$ produced in this way is the same as the desired $\{r_{R\star}(\omega) \mid r_R \in \rcal_{\langle S \rangle}\} \cap \rho_{\tcal}(Z)$.
Albeit using generators is conceptually simpler than using the generated subgroup, generators-based approaches are incremental in nature, leaving no room for parallelization.
Contrarily, when the subgroup is already computed, the subsequent orbit computation can be fully parallelized.
Therefore, in light of parallel computing, a generators-based approach might not always be a clear winner in practice.

\section{Computational Complexity Analysis}
\label{sec:comp-complexity-analysis}

Following the algorithmic roadmap (Section~\ref{sec:algorithmic-roadmap}), we analyze the complexity of each step and discuss parallelization.
Step~1 is outlined by Figure~\ref{fig:flow-of-computing-basis-of-translations-in-s}, having six substeps assuming the group-based approach.
Substeps~2,4,5 can be parallelized.%
\vspace{-0.1in}
\begin{enumerate}[label=$\blacksquare~$, itemsep=-0.03in]
\item Substep~1: use an off-the-shelf implementation to compute the permutation subgroup $\pcal_{\langle S \rangle} = \langle \pcal_S \rangle$, which can be done in $O(|\langle \pcal_S \rangle|)$ time.
\item Substep~2: generate $\ncal_S^+ = (\ncal_S)^{\pcal_{\langle S \rangle}}$ by $O(|\ncal_S||\pcal_{\langle S \rangle}|)$ conjugations.
\item Substep~3: first run the boolean version of Gaussian Elimination ($\mbox{GE}_b$) to get a basis of $\ncal_{\langle S \rangle}$ from its generating set $\ncal_S^+$; then generate $\ncal_{\langle S \rangle} = \langle \ncal_S^+ \rangle$ by enumerating all (boolean) linear combinations of the basis.
The former can be done by $O(n^2|\ncal_S^+|)$ bit-operations in a sequential-$\mbox{GE}_b$ and can be reduced to $O(n^2+n\log_2|\ncal_S^+|)$ bit-operations in a parallel-$\mbox{GE}_b$~\cite{KoccA1991};
the latter can be done in $O(2^d)$ time where $d$ is the size of the basis.

\underline{In comparison,} the generators-based approach replacing Substeps~1--3 can be done in $k\cdot O(n^2\cdot n|\pcal_S|) \leq O(n^4|\pcal_S|)$ assuming $k$ sequential-$\mbox{GE}_b$ runs.

\item Substep~4: generate $\rcal_{\langle S \rangle} = \ncal_{\langle S \rangle}\circ \pcal_{\langle S \rangle}$ by $O(|\ncal_{\langle S \rangle}||\pcal_{\langle S \rangle}|)$ compositions.
\item Substep~5: generate $\tcal_S^+ = (\tcal_S)^{\rcal_{\langle S \rangle}}$ by $O(|\tcal_S||\rcal_{\langle S \rangle}|)$ conjugations.
\item Substep~6: run the LLL algorithm to get a basis of $\tcal_{\langle S \rangle}$ from its generating set $\tcal_S^+$ in $O(|\tcal_S^+|^5n\log^3b)$ time (where $b$ is the largest Euclidean norm of the input vectors)~\cite{LenstraLL1982}.

\underline{In comparison,} the generators-based approach replacing Substeps~4--6 can be done in $k\cdot O(|\rcal_S|^5n^6\log^3b)$ assuming $k$ LLL runs (an upper bound on $k$ is unclear in this case).

\end{enumerate}
Step~2 involves basic matrix operations (\eg multiplication, inversion), which has a complexity of $O(2m^2n+m^3)$.
\underline{Note:} for a given generating set $S$, Steps 1 and 2 (or Stage 1) have nothing to do with the subset $Z$, and are considered one-time pre-computation and cached in the computer memory.

Entering Stage~2, Step~3 involves basic matrix-vector multiplications with a complexity of $O(mn)$ for every $x \in Z$.
So, the total of this step is $O(mn|Z|)$, but the computation of $|Z|$ projections can be parallelized.
Step~4 executes the while-loop whose complexity is dominated by $c\cdot O(|\rcal_{\langle S \rangle}|)$ projected rotations where $c = |Z/\langle S \rangle|$; inner loop over $\rcal_{\langle S \rangle}$ is parallelable.
The variant to this step (Section~\ref{sec:variant-to-step-4}) can be done in $\leq \sum_\omega k_\omega\cdot O(|C_\omega||\rcal_S|)$ time (sum over $c$ terms) assuming $k_\omega$ (non-parallelable) iterations for each $\omega$.
The outer while-loop can be parallelized over $Z$, but not over $c$ representatives from $Z/\langle S \rangle$; hence, if there are abundant computing resources, parallelization is preferred.
Step~5 is simply an $O(|Z|)$ labeling procedure and can certainly be parallelized.

\section{Conclusions, Applications, and Future Work}
\label{sec:conclusions-applications-future-work}

In this paper, we present a specialized algorithm to solve our restricted orbit computation problem under the isometry subgroup action $\langle S \rangle \curvearrowright \ambsp$, with the generating set $S \subseteq \iz$ being both finite and \emph{atomic}.
The essence of the algorithm is to leverage the \emph{semidirect-product decomposition} of the acting subgroup $\langle S \rangle$, and our newly introduced notion of an \emph{atomically generated subgroup} is key for such subgroup to inherit the global semidirect-product structure from $\iz$.
According to the decomposition structure, our algorithm takes two major stages---considering \emph{translation equivalence} first and \emph{rotation equivalence} in succession---to eventually reach an algorithmic implementation of an \emph{orbit-labeling map} $\lambda$.
%The map $\lambda$ produces a label for every point in the input space such that points are labeled same if and only if they are in the same orbit.
From $\lambda$, we can precisely recover the desired orbit-partition restricted to any finite input space $Z$.
Our specialized algorithm is designed in a way that exploits parallel computing in many of its subroutines, so it outperforms more generic and/or non-parallelable approaches (if any).

Besides its algorithmic merit, solving our restricted orbit computation problem is useful in many applications.
For example, we can make (and then learn) computational abstractions of music composition in the form of different sets of orbits~\cite{YuMV2019}, since many music transformations are isometries~\cite{Tymoczko2010}, \eg music transpositions are translations, melodic inversions are negations, and harmonic inversions are permutations.
Similar techniques for computational abstraction may be used to discover isometry-induced symmetries for biological data, \eg single-cell RNA sequencing data from the developing mouse retina~\cite{YuVS2019,ClarkSSCDSSHRJRO2019}.

To explore more applications in the future, we are interested in extending the main results here for Euclidean isometries acting on $\ambsp$ to hyperbolic isometries acting on hyperbolic spaces, since not only can we possibly leverage a similar semidirect-product decomposition, but also hyperbolic spaces may be useful in modeling some data spaces in the real world, \eg human olfactory space~\cite{ZhouSS2018}.
%By further learning computational abstractions of this space algorithmically, we hope to discover new understandings of human odorant perception as a potential application.

%%%%%%%%%%%%%%%%%%%%%%%
%% Elsevier bibliography styles
%%%%%%%%%%%%%%%%%%%%%%%
%% To change the style, put a % in front of the second line of the current style and
%% remove the % from the second line of the style you would like to use.
%%%%%%%%%%%%%%%%%%%%%%%

%% Numbered
%\bibliographystyle{model1-num-names}

%% Numbered without titles
%\bibliographystyle{model1a-num-names}

%% Harvard
%\bibliographystyle{model2-names.bst}\biboptions{authoryear}

%% Vancouver numbered
%\usepackage{numcompress}\bibliographystyle{model3-num-names}

%% Vancouver name/year
%\usepackage{numcompress}\bibliographystyle{model4-names}\biboptions{authoryear}

%% APA style
%\bibliographystyle{model5-names}\biboptions{authoryear}

%% AMA style
%\usepackage{numcompress}\bibliographystyle{model6-num-names}

%% `Elsevier LaTeX' style
\bibliographystyle{elsarticle-num}
%%%%%%%%%%%%%%%%%%%%%%%
\section*{References}
\bibliography{abrv,conf_abrv,isom-orbits}

\renewcommand{\theHsection}{A\arabic{section}}
\appendix

\newpage
\section{Mathematical Notations}
\label{app:math-notations}

\begin{table}[h!]
\centering
\begin{tabular}{rr}
    {\small{\textit{Notation}}} & {\small \textit{Meaning}} \\
    \hline
    $G$ & a group \\
    $X$ & a set \\
    $G \curvearrowright X$ & a $G$-action on $X$ \\
    $A_k \cdots A_1$ & the product of $A_k, \ldots, A_1 \subseteq G$ \\
    $[A_2A_1]$ & the (inner) semidirect product of $A_2, A_1 \leq G$ \\
    $\left[A_k  \cdots  A_1\right]$ & $\left[ A_k \left[ A_{k-1} \cdots A_1 \right] \right]$ recursively \\
    $G = \left[A_k  \cdots  A_1\right]$ & an (inner) semidirect-product decomposition of $G$ \\
    $\iz$ & the group of isometries of $\ambsp$ \\
    $\tz$ & the group of translations of $\ambsp$ \\
    $\rz$ & the group of (generalized) rotations of $\ambsp$ \\
    $\nz$ & the group of (coordinate-wise) negations of $\ambsp$ \\
    $\pz$ & the group of (coordinate-wise) permutations of $\ambsp$ \\
    $t_v: \ambsp \to \ambsp$ & a translation, $t_v(x) := x + v$ \\
    $r_R: \ambsp \to \ambsp$ & a (generalized) rotation, $r_R(x) := Rx$ \\
    $\tcal_S$ & $S \cap \tz$, for any $S \subseteq \iz$ \\
    $\rcal_S$ & $S \cap \rz$, for any $S \subseteq \iz$ \\
    $\ncal_S$ & $S \cap \nz$, for any $S \subseteq \iz$ \\
    $\pcal_S$ & $S \cap \pz$, for any $S \subseteq \iz$ \\
    $\cdot^b: A \to A$ & the conjugation function $a^b := bab^{-1}$ \\
%    $\idty_V$ & the identity function on a set $V$ \\
    $q$ & orbit-quotient map \\
    $\lambda~(= \overline{\lambda} \circ q)$ & orbit-labeling map ($\overline{\lambda}$ is injective) \\
    $\rho~(= \overline{\rho}\circ q)$ & orbit-representative map ($\overline{\rho}$ is a section of $q$) \\
    $B^\dagger~(= (B^\top B)^{-1}B^\top)$ & the pseudoinverse of a matrix $B$ (full column rank) \\
    $\lfloor \cdot \rfloor: \reals^m \to \integers^m$ & the (coordinate-wise) floor function \\
    $\rho_\tcal: \ambsp \to \ambsp$ & $\rho_\tcal(x) := x - B\lfloor B^\dagger x\rfloor$, an orbit-representative map \\
    $g_\star$ & projected-$g$ function
  \end{tabular}
\caption*{}
\end{table}

\newpage
\section{Mathematical Proofs}
\label{app:math-proofs}
\renewcommand{\appendixname}{}

\subsection{Theorem~\ref{thm:distinguishing-property-atomically-generated-subgroup-k-ary}}
\label{app:distinguishing-property-atomically-generated-subgroup-k-ary}
\begin{proof}
We prove by induction on $k$.
First of all, the base case (\ie $k=2$) is Theorem~\ref{thm:distinguishing-property-atomically-generated-subgroup-2-ary}.
Now assuming Theorem~\ref{thm:distinguishing-property-atomically-generated-subgroup-k-ary} holds for all $k-1$, we show that it also holds for $k$.

Let $H = \left[ A_{k-1}\cdots A_1 \right]$, then by the recursive definition of the bracket notation, $\ginner = \left[ A_k \left[A_{k-1} \cdots A_1\right] \right]$.
Therefore, $G = \left[ A_kH \right]$.
By Theorem~\ref{thm:distinguishing-property-atomically-generated-subgroup-2-ary},
\begin{gather}
\label{eqn:top-bin-level-decomp-temp}
\langle S \rangle = \left[(\acal_k)_{\langle S \rangle}{\hcal}_{\langle S \rangle} \right] \quad \mbox{ where } \\
\label{eqn:top-bin-level-decomp-1-temp}
(\acal_k)_{\langle S \rangle} = \langle (\acal_k)_S^+ \rangle = \varphi_{A_k}(\langle S \rangle) \\
\label{eqn:top-bin-level-decomp-2-temp}
\hcal_{\langle S \rangle} = \langle \hcal_S \rangle = \varphi_{H} (\langle S \rangle).
\end{gather}
In the above, $(\acal_k)_S^+ := ((\acal_k)_S)^{\hcal_{\langle S \rangle}}$ and notably, $\hcal_S = \hcal_S^+$.

Now we zoom into $H = \left[A_{k-1}\cdots A_1\right]$.
First, since $S$ is an atomic subset of $G$, \ie $S \subseteq \atomicsuper$, then $\hcal_S = S \cap H \subseteq (\atomicsuper)\cap H = \{e\} \cup A_{k-1}\cup \cdots \cup A_1 = A_{k-1}\cup \cdots \cup A_1$, which implies that $\hcal_S$ is an atomic subset of $H$ with respect to the $(k-1)$-ary semidirect-product decomposition of $H$.
Then, applying the induction hypothesis,
\begin{gather*}
\langle \hcal_S \rangle = \left[ (\acal_{k-1})_{\langle \hcal_S \rangle} \cdots (\acal_1)_{\langle \hcal_S \rangle} \right] \quad \mbox{ where } \\
(\acal_j)_{\langle \hcal_S \rangle} = \langle (\acal_j)_{\hcal_S}^+ \rangle = \varphi_{A_j}(\langle \hcal_S \rangle) \quad \mbox{ for any } j \in \{1, \ldots, k-1\}.
\end{gather*}
In the above, $(\acal_j)_{\hcal_S}^+ := ((\acal_j)_{\hcal_S})^{(\acal_{j-1})_{\langle \hcal_S \rangle} \cdots (\acal_1)_{\langle \hcal_S \rangle}}$.
By Equation~\eqref{eqn:top-bin-level-decomp-2-temp}, we observe that $\langle \hcal_S \rangle = \hcal_{\langle S \rangle}$, thus, for $j \in \{1, \ldots, k-1\}$,
\begin{align*}
(\acal_j)_{\langle \hcal_S \rangle} &= (\acal_j)_{\hcal_{\langle S \rangle}} = \langle S \rangle \cap H \cap A_j = \langle S \rangle \cap A_j = (\acal_j)_{\langle S \rangle} \\
\varphi_{A_j}(\langle \hcal_S \rangle) &= \varphi_{A_j}(\hcal_{\langle S \rangle}) = \varphi_{A_j}((\acal_k)_{\langle S \rangle}\hcal_{\langle S \rangle}) = \varphi_{A_j}(\langle S \rangle) \\
(\acal_j)_{\hcal_S} &= S \cap H \cap A_j = S \cap A_j = (\acal_j)_S.
\end{align*}
Therefore, we can rewrite the above expression for $\langle \hcal_S \rangle$ as follows:
\begin{gather*}
\hcal_{\langle S \rangle} = \left[ (\acal_{k-1})_{\langle S \rangle} \cdots (\acal_1)_{\langle S \rangle} \right] \quad \mbox{ where } \\
(\acal_j)_{\langle S \rangle} = \langle (\acal_j)_S^+ \rangle = \varphi_{A_j}(\langle S \rangle) \quad \mbox{ for any } j \in \{1, \ldots, k-1\}.
\end{gather*}
In the above, $(\acal_j)_S^+ := ((\acal_j)_S)^{(\acal_{j-1})_{\langle S \rangle} \cdots (\acal_1)_{\langle S \rangle}}$.

Plugging the above expression for $\hcal_{\langle S \rangle}$ in Expressions~\eqref{eqn:top-bin-level-decomp-temp}--\eqref{eqn:top-bin-level-decomp-2-temp}, we have
\begin{align*}
\langle S \rangle = \left[ (\acal_k)_{\langle S \rangle} \left[ (\acal_{k-1})_{\langle S \rangle} \cdots (\acal_1)_{\langle S \rangle} \right] \right] = \left[ (\acal_k)_{\langle S \rangle} \cdots (\acal_1)_{\langle S \rangle} \right],
\end{align*}
%That is, $\sgsdpd{\langle S \rangle}$
where for any $j \in \{1, \ldots, k\}$, $(\acal_j)_{\langle S \rangle} = \langle (\acal_j)_S^+ \rangle = \varphi_{A_j}(\langle S \rangle)$.

We finally check that the augmented generating sets are also consistent.
By definition, $(\acal_k)_S^+ = ((\acal_k)_S)^{\hcal_{\langle S \rangle}} = ((\acal_k)_S)^{(\acal_{k-1})_{\langle S \rangle} \cdots (\acal_1)_{\langle S \rangle}}$, which is indeed consistent with the other formulae $(\acal_j)_S^+ = ((\acal_j)_S)^{(\acal_{j-1})_{\langle S \rangle} \cdots (\acal_1)_{\langle S \rangle}}$ for any $j \in \{1, \ldots, k-1\}$.
This completes the proof.
\end{proof}

\subsection{Theorem~\ref{thm:rep-qmap-def-alt1}}
\label{app:rep-qmap-def-alt1}
\begin{proof}
Let $\overline{\rho}: X/G \to X$ be a function defined by $\overline{\rho}(G\cdot x) := \rho(x)$.
Pick any $G\cdot x$, $G\cdot x' \in X/G$, and apply the two directions of \ref{item1:three-properties-of-rep-qmap}.
The fact that $G\cdot x = G\cdot x' \implies \rho(x) = \rho(x')$ indicates that $\overline{\rho}$ is well-defined; the fact that $\rho(x) = \rho(x') \implies G\cdot x = G\cdot x'$ indicates that $\overline{\rho}$ is injective.
Further, for any $G \cdot x \in X/G$, $q \circ \overline{\rho}(G\cdot x) = q(\rho(x)) = G\cdot \rho(x) = G\cdot x$, where we used \ref{item2:three-properties-of-rep-qmap} in the last equality.
This implies that $q \circ \overline{\rho} = \idty_{_{X/G}}$.
Therefore, $\overline{\rho}$ is a section of $q$.
Finally, it is clear that $\rho = \overline{\rho} \circ q$ is an orbit-representative map of $G \curvearrowright X$.
\end{proof}

\subsection{Theorem~\ref{thm:rep-qmap-def-alt2}}
\label{app:rep-qmap-def-alt2}
\begin{proof}
Let $\overline{\rho}: X/G \to X$ be a function defined by $\overline{\rho}(G\cdot x) := \rho(x)$.
Following the same argument from the proof of Theorem~\ref{thm:rep-qmap-def-alt1}, \ref{item1:three-properties-of-rep-qmap} alone guarantees that $\overline{\rho}$ is well-defined and injective.
Further, for any $G \cdot x \in X/G$,
\begin{align*}
\overline{\rho} \circ q \circ \overline{\rho}(G\cdot x) = \overline{\rho} \circ q \circ \overline{\rho} \circ q(x) = \rho \circ \rho(x) = \rho(x) = \overline{\rho}(G\cdot x),
\end{align*}
where the second last equality is \ref{item3:three-properties-of-rep-qmap}, \ie $\rho$ is an idempotent.
Since $\overline{\rho}$ is injective, then the above equation further implies that $q \circ \overline{\rho}(G\cdot x) = G\cdot x$ for any $G \cdot x \in X/G$, \ie $q \circ \overline{\rho} = \idty_{_{X/G}}$.
Therefore, $\overline{\rho}$ is a section of $q$.
Finally, it is clear that $\rho = \overline{\rho} \circ q$ is an orbit-representative map of $G \curvearrowright X$.
\end{proof}

\end{document}